\documentclass[12pt,a4paper]{article}
\usepackage{threeparttable}
\usepackage[caption = false,labelformat=empty]{subfig}
\usepackage{float}
\usepackage[skip=1pt,font={normalsize},singlelinecheck=on]{caption}
\usepackage{amsmath,amsfonts,amsthm,amssymb}
\usepackage{dsfont}
\usepackage{natbib}
\usepackage{mathrsfs}
\usepackage{pxfonts}
\usepackage[utf8]{inputenc}
\usepackage[T1]{fontenc}
\usepackage[colorlinks=true,linkcolor=blue, citecolor=blue,urlcolor=blue]{hyperref}
\usepackage{csquotes}
\usepackage{float}
\usepackage{graphicx}
\usepackage{multirow}
\usepackage[caption=false]{subfig}
\theoremstyle{plain}
\newtheorem{theorem}{Theorem}[section]
\newtheorem{lemma}[theorem]{Lemma}

\newtheorem{proposition}[theorem]{Proposition}

\theoremstyle{definition}

\newtheorem{remark}[theorem]{Remark}

\begin{document}

\title{An effective estimation of multivariate density functions using extended-beta kernels with Bayesian adaptive bandwidths}

\author{Sobom M. \textsc{Som\'e}\thanks{(\textbf{Corresponding Author}) Universit\'{e} Thomas SANKARA, Laboratoire Sciences et Techniques, 12 BP 417 Ouagadougou 12, Burkina Faso.              \href{mailto:sobom.some@uts.bf}{sobom.some@uts.bf}; \& Universit\'{e} Joseph KI-ZERBO, Laboratoire d'Analyse Num\'erique d'Informatique et de BIOmath\'ematique, 03 BP 7021 Ouagadougou 03, Burkina Faso. \href{mailto:sobom.some@univ-ouaga.bf}{sobom.some@univ-ouaga.bf}}, \;
	C{\'e}lestin C. \textsc{Kokonendji}\thanks{Universit\'e Marie \& Louis Pasteur, Laboratoire de math{\'e}matiques de Besan\c{c}on UMR 6623 CNRS-UMLP, 16 route de Gray, 25030 Besan{\c c}on Cedex, France. \href{mailto:celestin.kokonendji@univ-fcomte.fr}{celestin.kokonendji@univ-fcomte.fr}; \&  Universit\'e de Bangui, Laboratoire de math\'ematiques et connexes de Bangui, B.P. 908 Bangui, Centrafrique. \href{mailto:kokonendji@gmail.com}{kokonendji@gmail.com}} \;\\  and \;
	Francial G.B. \textsc{Libengu{\'e} Dob{\'e}l{\'e}-Kpoka}\thanks{ Universit\'e de Bangui, Laboratoire de math\'ematiques et connexes de Bangui, BP 908 Bangui, Centrafrique.  \href{mailto:libengue@gmail.com }{libengue@gmail.com }} 
}
\maketitle

\begin{abstract}
	\noindent
Multivariate kernel density estimations have received much spate of interest. In addition to conventional methods of (non-)classical associated-kernels for (un)bounded densities and bandwidth selections, the multiple extended-beta kernel (MEBK) estimators with Bayesian adaptive bandwidths are invested to gain a deeper and better insight into the estimation of multivariate density functions. Being unimodal, 
the univariate extended-beta smoother has an adaptable compact support which is suitable for each dataset, always limited. The support of the density MBEK estimator can be known or estimated by extreme values. Thus, asymptotical properties for the (non-)normalized estimators are established. Explicit and general choices of bandwidths using the flexible Bayesian adaptive method are provided. Behavioural analyses, specifically undertaken on the sensitive edges of the estimator support, are studied and compared to Gaussian and gamma kernel estimators. Finally, simulation studies and three  applications on original and usual real-data sets of the proposed method yielded very interesting advantages with respect to its flexibility as well as its universality.	
\end{abstract}


\section{Introduction and motivation}\label{sec_Introduction}

Since the pioneers \cite{Rosenblatt1956} and \cite{Parzen1962}, the symmetric kernel density estimations have undergone several evolutions. For instance, the multivariate case was first considered in a single bandwidth by \cite{Cacou66} and with a vector of bandwidth by \cite{Epan69} who provided an optimal symmetric kernel. It has equally been addressed by \cite{Silver86}, \cite{ScoWan91}, \cite{TerSco92}, \cite{WanJon95}, \cite{MarshallH2010} and \cite{KangEtAl18}. Owing to equivalence of symmetric continuous kernels, much ink has been spilled upon the bandwidth matrix selection; see, e.g., \cite{Duin1976}, \cite{DH05} and \cite{Zoug14}. Another path of evolution concerns the asymmetric kernels for solving the so-called edge effects in the estimation of densities with bounded support at least on one side; see, e.g., \cite{Chen99,Chenn2000b} for univariate beta and gamma kernels, \cite{BouRom10} and \cite{SomeEtAl2024} in the corresponding multivariate contexts. For other asymmetric ones, one can also refer to \cite{JinKaw2003}, \cite{MarchantEtAl2013},  \cite{HirukawaS2015}, \cite{FunKaw15} and \cite{LK17}.

Basically, it is now desirable to unify all these kernels, continuous (a)symmetric or not with the discrete cases, under the so-called associated-kernels for estimating (continuous) density or (discrete) probability mass function $f$ on its $d$-dimensional support $\mathbb{T}_d\subseteq\mathbb{R}^d$. In this regard, an associated-kernel can be defined as follows. Let $\boldsymbol{x}=(x_1,\ldots,x_d)^\top\in\mathbb{T}_d$ be the point of estimation of $f$, and $\mathbf{H}$ a ($d\times d$)-symmetric and positive definite matrices (referred to as bandwidth matrices). A multivariate associated-kernel function $\mathbf{K}_{\boldsymbol{x},\mathbf{H}}(\cdot)$ parametrized by $\boldsymbol{x}$ and $\mathbf{H}$ with support $\mathbf{S}_{\boldsymbol{x},\mathbf{H}}\subseteq\mathbb{R}^d$ satisfies the following conditions:
\begin{equation}\label{MAK}
(i)\; \boldsymbol{x}\in\mathbf{S}_{\boldsymbol{x},\mathbf{H}};\;(ii)\;
\mathbb{E}(\mathcal{Z}_{\boldsymbol{x},\mathbf{H}})=\boldsymbol{x}+\mathbf{A}(\boldsymbol{x},\mathbf{H});\;(iii)\;Var(\mathcal{Z}_{\boldsymbol{x},\mathbf{H}})=\mathbf{B}(\boldsymbol{x},\mathbf{H}).
\end{equation}
The $d$-dimensional random vector $\mathcal{Z}_{\boldsymbol{x},\mathbf{H}}$ has a probability density mass function (pdmf) $\mathbf{K}_{\boldsymbol{x},\mathbf{H}}(\cdot)$ such that vector $\mathbf{A}(\boldsymbol{x},\mathbf{H})\to \mathbf{0}$ and positive definite matrix $\mathbf{B}(\boldsymbol{x},\mathbf{H})\to\mathbf{0}_d$ as $\mathbf{H}\to\mathbf{0}_d$, where $\mathbf{0}$ and $\mathbf{0}_d$ denote the $d$-dimensional null vector and the $d\times d$ null matrix, respectively. Thus, the usual associated-kernel estimator $\widehat{f}(\boldsymbol{x})$ of $f(\boldsymbol{x})$ is defined as 
$\widehat{f}(\boldsymbol{x})=(1/n)\sum_{i=1}^n\mathbf{K}_{\boldsymbol{x},\mathbf{H}}(\mathbf{X}_i)$ for an iid-sample $(\mathbf{X}_1,\ldots,\mathbf{X}_n)$ from the unknown $f$.
One can refer to \cite{KS18} for some properties, where a construction and illustrations of the so-called full, Scott and diagonal bandwidth matrices for general, multiple and classical associated kernels, are respectively provided. See also \cite{AboubacarK24}, \cite{Esstafa22,Esstafa23b},  \cite{SengaDurr24} and \cite{SomeEtAl2023} for additional  properties and references therein.

Within this framework, there are two common techniques for generating the well-known multivariate associated-kernels satisfying \eqref{MAK}:
\begin{equation}\label{AKcNc}
(i)\;\mathbf{K}_{\boldsymbol{x},\mathbf{H}}(\cdot)=(\det\mathbf{H})^{-1/2}\mathbf{K}\left(\mathbf{H}^{-1/2}(\boldsymbol{x}-\cdot)\right)\;\;\mathrm{and}\;\;
(ii)\;\mathbf{K}_{\boldsymbol{x},\mathbf{H}}(\cdot)=\prod_{k=1}^dK_{x_k,h_{kk}}^{[k]}(\cdot),
\end{equation}
where $\mathbf{K}(\cdot)$ refers to a $d$-variate (classical) symmetric kernel, $K_{x_k,h_{kk}}^{[k]}(\cdot)$ corresponds to a univariate (non-classical) associated-kernel with $\mathbf{H}=(h_{jk})_{j,k=1,\ldots,d}$, and $\det\mathbf{H}$ stands for the determinant of $\mathbf{H}$. The first multivariate kernel displayed in Part (i) of \eqref{AKcNc} can be called \textit{classical multivariate associated-kernel}, as for a given symmetric kernel $\mathbf{K}$ on $\mathbf{S}_d\subseteq\mathbb{R}^d$, one has $\mathbf{S}_{\boldsymbol{x},\mathbf{H}}=\boldsymbol{x}-\mathbf{H}^{1/2}\mathbf{S}_d$, $\mathbf{A}(\boldsymbol{x},\mathbf{H})=\mathbf{0}$ and $\mathbf{B}(\boldsymbol{x},\mathbf{H})=\mathbf{H}^{1/2}\boldsymbol{\Sigma}\mathbf{H}^{1/2}$, where $\boldsymbol{\Sigma}$ is a covariance matrix not depending on $\mathbf{x}$ and $\mathbf{H}$. 
The second kernel exhibited in Part (ii) of (\ref{AKcNc}) is commonly labelled as \textit{multiple associated-kernels} with $d$-univariate kernels $K_{x_k,h_{kk}}^{[k]}(\cdot)$ which can be either classical or non-classical. One of the merits of the multiple associated-kernels with diagonal bandwidth matrices  resides in mixing two or more different sources of the univariate associated-kernels following the support $\mathbb{T}_1^{[k]}\subseteq\mathbb{R}$ of the $k$th component of $\boldsymbol{x}$. From this perspective, exploring univariate setup provides a better understanding and deeper  vizualisations of the problem.
\begin{figure}[!hbtp]
	\centering
	\subfloat[]{%
		\resizebox*{11cm}{!}{\includegraphics[width=11.cm, height=9cm]{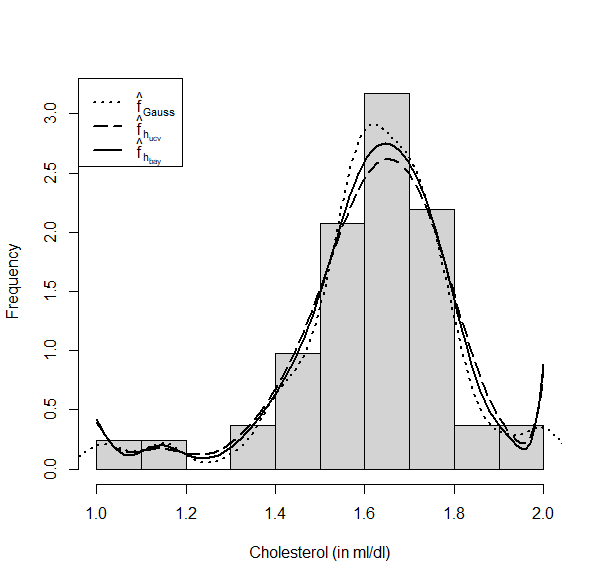}}}	
	\caption{Histogram with its corresponding smoothings of the cholesterol data on $[1,2]$ from Table \ref{univ_data} using univariate Gaussian and extended-beta kernels with both cross-validation and Bayes selectors of bandwidths. }\label{Fig:cholesterol1}
\end{figure} 

In the current paper and for simplicity in the univariate situation, we consider the kernel estimating problem of an unknown density function $f_1$ on its known or unknown continuous support $\mathbb{T}_1\subseteq\mathbb{R}$. Among many continuous associated-kernels depending on the support $\mathbb{T}_1$, which of them is more interesting and versatile for (un)bounded $\mathbb{T}_1:=[a,b]$ with $a<b$? How to estimate the support $\mathbb{T}_1$ when it is unknown too? How to select easily the smoothing parameter? Obviously, we exclude all classical or symmetrical (associated-)kernels which have been only designed for fully unbounded $\mathbb{T}_1=\mathbb{R}$, but never for compact $\mathbb{T}_1=[a,b]$. For instance, one can take a look at Figure \ref{Fig:cholesterol1} which may serve as an eye opening indicating the crux of this problem, mainly at the edges. This problem shall be subsequently analyzed then solved in Section \ref{sec_application}. In this respect, Gaussian kernel is not suitable for estimating density with compact support, while extended-beta kernel can work for any support. See also Figure \ref{Fig:marks}. Note that no datum value is found at infinity. Therefore, the observed extreme values are always finite which lead to estimating the support $\mathbb{T}_1$ through percentiles. As far as we know, the only solution to this task lies in considering an extension of the standard beta kernel on the compact support $\mathbb{S}_1=[a,b]$, namely extended-beta kernel of \cite{LK17}. Indeed, we shall demonstrate that it is adaptable to estimate all density functions on a given support (un)bounded  $\mathbb{T}_1\subseteq\mathbb{R}$ using appropriate handlings of the corresponding extended-beta kernel estimator at the edges.

The central objective of this work is to provide an effective kernel estimation of any multivariate density functions. We first introduce the multiple extended-beta kernels in the sense of Part (ii) of \eqref{AKcNc} highlighting some of its  (asymptotical) properties. Next, we set forward explicit bandwidths through the Bayesian adaptive method before undertaking numerical illustrations and discussions. Notably, the rest of the paper is structured as follows. Section \ref{sec_MEBKestimator} is devoted to the multiple extended-beta kernel estimators for density functions on compact support. Section \ref{sec_Main properties} foregrounds main asymptotical properties of the corresponding (non-)normalized estimators. Section \ref{sec_BayesianAdaptive} presents  explicit bandwidths by using the Bayesian adaptive approach. In Section \ref{sec_simulation}, the practical performances of our proposed method are investigated by simulation studies. Section \ref{sec_application} is dedicated to three original and classic real-data applications with discussions. Section \ref{sec_conclusion} concludes the paper with final remarks. Proofs of all results are summarized and incorporated into Appendix of Section \ref{sec_appendix proofs}.

\section{Multiple extended-beta kernel estimators}\label{sec_MEBKestimator}

Let $\mathbf{X}_{1},\ldots,\mathbf{X}_{n}$  be iid $d$-variate random variables with an unknown density $f$ on its compact support $\mathbb{T}_{d}:=\times_{j=1}^{d}[a_{j}, b_{j}]$ for $d \geq 1$ and $a_{j} < b_{j}$. The multiple extended-beta kernel (MEBK) estimator $\widehat{f}_{n}$ of $f$ is defined as
\begin{eqnarray}
	\widehat{f}_{n}(\boldsymbol{x})=\frac{1}{n} \sum_{i=1}^{n} \prod_{j=1}^{d}{EB}_{x_j,h_j,a_j,b_j}(X_{ij}), \label{betaestimator}
\end{eqnarray}
where $\boldsymbol{x}=(x_{1},\ldots,x_{d})^{\top} \in \mathbb{T}_{d}=\times_{j=1}^{d}[a_{j}, b_{j}]$, $\mathbf{X}_{i}=(X_{i1},\ldots,X_{id})^{\top}$, $i=1,\ldots,n$, $\boldsymbol{h}=(h_{1},\ldots,h_{d})^{\top}$ is the vector of the bandwidth parameters with $h_{j}=h_{j}(n)>0$ and $h_{j}\rightarrow 0$ as $n\rightarrow \infty$ for $j=1,\ldots,d$. The function $EB_{x,h,a,b}(\cdot)$ is the (standard) univariate extended-beta kernel defined on its compact support $\mathbb{S}_{x,h,a,b}=[a,b]=\mathbb{T}_1$ with $-\infty<a<b<\infty$, $x \in \mathbb{T}_1$ and $h>0$ such that
\begin{equation}\label{noyaugamma}
{EB}_{x,h,a,b}(u) = \frac {(u-a)^{(x-a)/\{(b-a)h\}}(b-u)^{(b-x)/\{(b-a)h\}}} {(b-a)^{1+1/h}B\{1+(x-a)/(b-a)h,1+(b-x)/(b-a)h\}}\mathbf{1}_{[a,b]}(u),
\end{equation}
where ${B(r,s)=\int_0^1 t^{r-1}(1-t)^{s-1}dt}$ is the usual beta function with $r>0$, $s>0$, and $\mathbf{1}_{A}$ denotes the indicator function in any given set $A$. 

The MEBK $\prod_{j=1}^{d}{EB}_{x_j,h_j,a_j,b_j}(\cdot)$ from \eqref{betaestimator} easily verifies all conditions \eqref{MAK}; in particular, one here has the $d$-vector  
\begin{equation}\label{eq:AxH}
\mathbf{A}(\boldsymbol{x}, \mathbf{H})=\left(\frac{(a_j+b_j-2x_j)h_j}{1+2x_j}\right)_{j=1,\ldots,d}^\top
\end{equation}
and the diagonal $d\times d$ matrix 
\begin{equation}\label{eq:BxH}
\mathbf{B}(\boldsymbol{x},\mathbf{H})= \mathbf{Diag}_d\left(\frac{\{(x_j-a_j+(b_j-a_j)h_j\}\{(b_j-x_j+(b_j-a_j)h_j\}h_j}{(1+2h_j)^2(1+3h_j)}\right)_{j=1,\ldots,d}
\end{equation}
from the univariate case of \cite{LK17}.

This extended-beta kernel (\ref{noyaugamma}) is the pdf of the extended-beta distribution on $[a,b]$, which can be denoted by $\mathcal{B}\left(c,d\right)$ with $c:=1+(x-a)/[(b-a)h]>1$ and $d:=1+(b-x)/[(b-a)h]>1$ as shape parameters under conditions of unimodaliy. In this case, its mode and dispersion parameter are indicated by $[(c-1)b+(d-1)a/(c+d-2)]=x$ and $1/(c+d-2)=h$, respectively. This kernel derives from a reparameterization in target $x$ and tuning parameter $h$, called the Mode-Dispersion method, of the unimodal extended-beta distribution $\mathcal{B}(c,d)$ on $[a,b]$ with $c>1$ and $d>1$; see \cite{LK17} for further details. Note that it is appropriate for any compact support of observations. For $a=0$ and $b=1$, it corresponds to the standard beta kernel of \cite{Chen99}; see also \cite{BertinKl2014}. 
\begin{figure}[!htbp]
		\centering
		\subfloat[(a)]{%
			\includegraphics[width=7.cm, height=7cm]{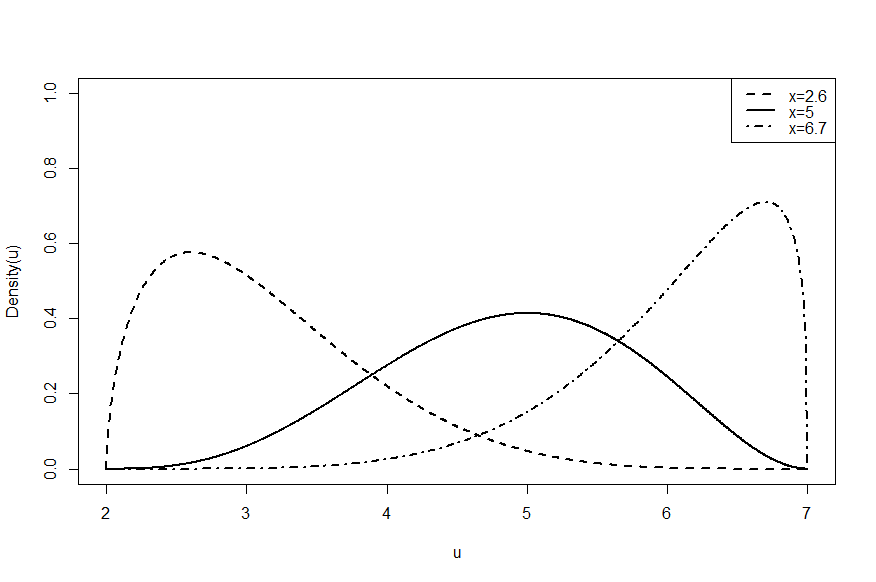}}
		\subfloat[(b)]{%
			\includegraphics[width=7.cm, height=7cm]{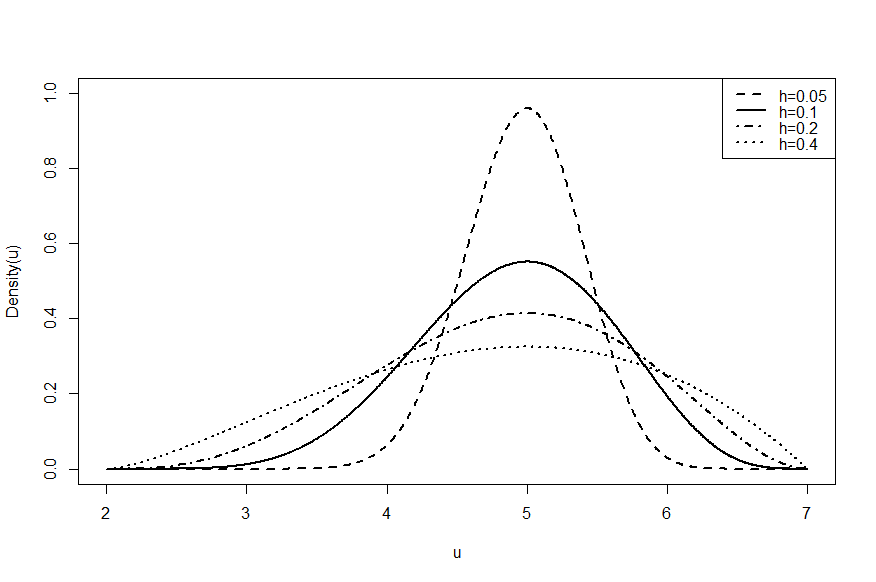}}
		\caption{Shapes of univariate extended-beta kernels with different targets $x$ and same smoothing parameter $h=0.2$ (a); with same target $x=5$ and different smoothing parameters (b).}\label{Fig:betakern}
	\end{figure}

Notice that this unimodal extended-beta distribution or kernel $EB_{x,h,a,b}$ in (\ref{noyaugamma}) is also known as the (four-parameter) modified beta-PERT\footnote{Project Evaluation and Review Technique} distribution from the minimum ($a$), maximum ($b$), mode ($m:=x$) and a fourth parameter ($\gamma:=1/h$); see, e.g., \cite{Clark1962}, \cite{Grubbs62} and also Figure \ref{Fig:betakern}. The parameter $\gamma:=1/h$ is the only shape in this context of reparametrization such that $h:=1/\gamma$ plays its crucial  role in the smoothing parameter or bandwidth; see Part (b) of  Figure \ref{Fig:betakern}. Basically, up to the shape parameter $\gamma$, this distribution continuously displays a frequency of values around the mode that drops softly in both directions, which decreases sharply from mode to extreme values. Estimations of the support $[a,b]$ (or possible ends) through extreme values can be easily performed through the use of percentiles. This refers to the fact that the parameters $a$ and $b$ are intuitive since they represent the possible interval of values occurence. One can refer in this regard to some softwares at the origin of this (modified) beta-PERT distribution in Operation Research for modeling activities through probabilistic approach; e.g., \cite{Vose2008}.
	
\section{Main properties of MEBK estimators}\label{sec_Main properties}

We need the following technical lemma for some below results and proofs.

\begin{lemma}\label{lemma}
The MEBK $\;\prod_{j=1}^{d}{EB}_{x_j,h_j,a_j,b_j}(\cdot)$ from \eqref{betaestimator} to \eqref{eq:BxH} is such that
$$\left|\left|\prod_{j=1}^{d}{EB}_{x_j,h_j,a_j,b_j}\right|\right|_2^2
=\prod_{j=1}^{d}
			\dfrac{(b_j-a_j)^{-1}{B}\left(1+\dfrac{2(x_j-a_j)}{(b_j-a_j)h_j},1+\dfrac{2(b_j-x_j)}{(b_j-a_j)h_j}\right)}
			{\left\lbrace {B}\left(1+\dfrac{(x_j-a_j)}{(b_j-a_j)h_j},1+\dfrac{(b_j-x_j)}{(b_j-a_j)h_j}\right)\right\rbrace ^{2}}.
$$
Furthermore, for some $\alpha_j>0$ ($j=1,\ldots,d$) and any $\boldsymbol{x}=(x_{1},\ldots,x_{d})^{\top} \in\mathbb{T}_d=\times_{j=1}^{d}[a_{j}, b_{j}]$, one has:
\begin{equation*}
||\mathbf{A}(\boldsymbol{x},\mathbf{H})||^2=\mathcal{O}\left(\prod_{j=1}^d h_j^{\alpha_j}\right),\;\;
\mathrm{trace}[ \mathbf{B}(\boldsymbol{x},\mathbf{H})]=\mathcal{O}\left(\prod_{j=1}^d h_j^{2\alpha_j}\right),
\end{equation*}
\begin{equation*}
\max_{\boldsymbol{y}\in\mathbb{T}_d}{EB}_{x_j,h_j,a_j,b_j}(\mathbf{y})=\mathcal{O}\left(\prod_{j=1}^d h_j^{-\alpha_j}\right)\;\;\mathrm{and}\;\;
\left|\left|\prod_{j=1}^{d}{EB}_{x_j,h_j,a_j,b_j}\right|\right|_2^2=\mathcal{O}\left(\prod_{j=1}^d h_j^{-\alpha_j}\right).
\end{equation*}
\end{lemma}

The next two assumptions are required  afterwards for asymptotic properties of the MEBK estimator in \eqref{betaestimator}.
	\begin{description}
		\item\label{a1} $({\bf a1})$ The unknown pdf $f$ belongs to the class $\mathscr{C}^2(\mathbb{T}_d)$ of twice continuously differentiable functions. 	
		\item\label{a2} ({\bf a2}) When $n\rightarrow \infty$, then $h_j:=h_j(n) \rightarrow 0$, $j=1,\ldots,d$ and $n^{-1}\displaystyle \prod_{j=1}^d h_j^{-1/2} \rightarrow 0$.		
	\end{description}	
\begin{proposition}\label{prop_BVf}
Under assumption $({\bf a1})$ on $f$, the MEBK estimator $\widehat{f}_{n}$ in \eqref{betaestimator} of $f$ satisfies		
\begin{align*}
Bias\left\lbrace\widehat{f}_{n}(\boldsymbol{x})\right\rbrace&=\sum_{j=1}^{d} \dfrac{(a_j+b_j-2x_j)h_j}{1+2h_j}\frac{\partial f(\boldsymbol{x})}{\partial x_{j}}+\dfrac{1}{2}\sum_{j=1}^{d}\left\{
			\left[ \dfrac{(a_j+b_j-2x_j)h_j}{1+2h_j}\right]^{2} \right.
			\\& \left.\quad + \dfrac{\{x_j-a_j+(b_j-a_j)h_j\}\{a_j-x_j+(b_j-a_j)h_j\}}{(1+2h_j)^{2}(1+3h_j)}\right\} \frac{\partial^{2}f(\boldsymbol{x})}{\partial x_{j}^{2}} +o\left(\sum_{j=1}^{d}h_{j}^2\right)
\end{align*}
for any $\boldsymbol{x}=(x_{1},\ldots,x_{d})^{\top} \in\mathbb{T}_d$. Moreover, if $({\bf a2})$ holds, then the variance is such that
\begin{align*}
Var\left\lbrace\widehat{f}_{n}(\boldsymbol{x})\right\rbrace&=\frac{1}{n}f(\boldsymbol{x})
\left|\left|\prod_{j=1}^{d}{EB}_{x_j,h_j,a_j,b_j}\right|\right|_2^2
			+o\left(\frac{1}{n}\prod_{j=1}^{d}h_{j}^{-1/2}\right).\label{VarEB}
\end{align*} 
\end{proposition}

According to some general limit results in \cite{Esstafa23b}, \cite{KL18} as well as \cite{KS18} for continuous associated-kernel estimators, MEBK estimator in \eqref{betaestimator} is pointwisely consistent and we also have its asymptotic normality.

\begin{proposition}\label{propo_L2asNorm}
If the sequences $\left(h_j(n)\right)_{n\geq 1}$ are selected  such that $nh_j^{\alpha_j}(n)\underset{n\to \infty}{\longrightarrow}\infty$ with $\alpha_j\geq 1$, for $j=1,\ldots,d$; then, for any $\boldsymbol{x}\in\mathbb{T}_d$,
$$\widehat{f}_{n}(\boldsymbol{x})\xrightarrow[n\to \infty]{\mathbb{L}^2,\:a.s.} f(\boldsymbol{x}),$$
where $` `\overset{\mathbb{L}^2,\:a.s.}{\longrightarrow}"$ stands for both $` `$mean square and almost surely convergences$"$. 
Furthermore, if $\left(h_j(n)\right)_{n\geq 1}$ are also chosen such that $\sqrt{n}h_j^{(3/2)\alpha_j}(n)\underset{n\to \infty}{\longrightarrow} 0$, for $j=1,\ldots,d$; then, for any $\boldsymbol{x}\in\mathbb{T}_d$,
$$\sqrt{n}\left(\prod_{j=1}^d h_j^{(1/2)\alpha_j}(n)\right)\left(\widehat{f}_{n}(\boldsymbol{x})-f(\boldsymbol{x})\right) \xrightarrow[n\to \infty]{\mathcal{L}} \mathcal{N}(0,f(\boldsymbol{x})\lambda_{\boldsymbol{x},\boldsymbol{\alpha}}),$$
where $` `\overset{\mathcal{L}}{\longrightarrow}"$ and $\mathcal{N}(0,\tau)$ indicate $` `$convergence in law$"$ and centered normal distribution with variance $\tau>0$, respectively, and 
\begin{equation}\label{lambda_xa}
\lambda_{\boldsymbol{x},\boldsymbol{\alpha}}=\lim_{n\to\infty}\left(\prod_{j=1}^d h_j^{\alpha_j}(n)\right)\left|\left|\prod_{j=1}^{d}{EB}_{x_j,h_j(n),a_j,b_j}\right|\right|_2^2.
\end{equation}
\end{proposition}

In what follows, according to the practical use and for given bandwidths, we always consider the standard normalized MEBK estimator $\widetilde{f}_{n}$ of $f$ defined from the unnormalized one $\widehat{f}_{n}$ of \eqref{betaestimator} as follows:
\begin{equation}\label{Cn_f}
\widetilde{f}_{n}(\boldsymbol{x}):=\frac{\widehat{f}_{n}(\boldsymbol{x})}{C_n}\;\;\;\mathrm{with}\;\;\;C_n:=\int_{\mathbb{T}_d}\widehat{f}_{n}(\boldsymbol{x})\,d\boldsymbol{x}>0.
\end{equation}
One can see in \cite{LK17} in addition to \cite{KS18} certain numerical results suggesting that the normalizing random variable $C_n$ fails to be equal to $1$ and is always around $1$.
\begin{proposition}\label{propo_Cn}
Under assumptions of Proposition \ref{propo_L2asNorm}, then $C_n \xrightarrow[n\to \infty]{\mathbb{L}^2} 1.$
\end{proposition}

Now, we deduce that $\widetilde{f}_{n}(\boldsymbol{x})$ is weakly consistent and also asymptotically normal.

\begin{proposition}\label{propo_L2Norm}
Under the same assumptions of Proposition \ref{propo_L2asNorm}, one has
$$\widetilde{f}_{n}(\boldsymbol{x})\xrightarrow[n\to \infty]{\mathbb{P}} f(\boldsymbol{x}),$$
where $` `\overset{\mathbb{P}}{\longrightarrow}"$ stands for  $` `$convergence in probability$"$. 
Moreover, one also has
$$\sqrt{n}\left(\prod_{j=1}^d h_j^{(1/2)\alpha_j}(n)\right)\left(\widetilde{f}_{n}(\boldsymbol{x})-f(\boldsymbol{x})\right) \xrightarrow[n\to \infty]{\mathcal{L}} \mathcal{N}(0,f(\boldsymbol{x})\lambda_{\boldsymbol{x},\boldsymbol{\alpha}})$$
with the same $\lambda_{\boldsymbol{x},\boldsymbol{\alpha}}$ provided in \eqref{lambda_xa}.
\end{proposition}

Finally, since $\mathbb{T}_d$ is a compact set, the next proposition holds that $\widetilde{f}_{n}$ globally outperforms $\widehat{f}_{n}$, in the sense of $L^2$ criterion.

\begin{proposition}\label{propo_Comparison}
Under assumptions of Proposition \ref{propo_L2asNorm}, for any $\varepsilon>0$, there exists $N\in\mathbb{N}$ such that for all $n\geq N$,
$$\mathbb{E}\left[\int_{\mathbb{T}_d}\left|\widetilde{f}_{n}(\boldsymbol{x})-f(\boldsymbol{x})\right|^2d\boldsymbol{x}\right]
<\mathbb{E}\left[\int_{\mathbb{T}_d}\left|\widehat{f}_{n}(\boldsymbol{x})-f(\boldsymbol{x})\right|^2d\boldsymbol{x}\right]
+\varepsilon.
$$
\end{proposition}

\section{Bayesian adaptive bandwidth selector}
	\label{sec_BayesianAdaptive}
	
In this section, we first provide the Bayesian adaptive approach in order to select the variable smoothing parameter, suitable for the MEBK  estimator \eqref{betaestimator} of an unknown pdf $f$. In this vein, we set forward two remarks including one on the nonparametric support estimation of $f$.
	
For a $d$-variate random sample $\mathbf{X}_{1},\ldots,\mathbf{X}_{n}$ drawn from $f$, we assign a random variable bandwidth vector $\boldsymbol{h}_{i}=\left(h_{ij}\right)^\top_{j=1,\ldots,d}$ with a prior distribution $\pi(\cdot)$ to each observation $\mathbf{X}_{i}$:	
	\begin{eqnarray}
		\widehat{f}_{n}(\boldsymbol{x})=\frac{1}{n} \sum_{i=1}^{n} \prod_{j=1}^{d}{EB}_{x_j,h_{ij},a_j,b_j}(X_{ij}). \label{Adapbetaestimator}
	\end{eqnarray}
Accordingly, we start with the formula for the leave-one-out kernel estimator of $f(\mathbf{X}_{i})$ deduced from \eqref{Adapbetaestimator} and built up considering the sample without the $i$-th observation,
	\begin{eqnarray}
		\widehat{f}_{-i}(\mathbf{X}_{i}\mid\boldsymbol{h}_{i})=\frac{1}{n-1} \sum_{j=1,j\neq i}^{n} \prod_{\ell=1}^{d}{EB}_{X_{i\ell},h_{i\ell},a_{\ell},b_{\ell}}(X_{j\ell}).\label{leaveoneout}
	\end{eqnarray}
From the Bayesian rule, it follows that the posterior distribution for each variable bandwidth vector $\boldsymbol{h}_{i}$ given $\mathbf{X}_{i}$ is expressed in terms of
	\begin{eqnarray}
		\pi(\boldsymbol{h}_{i}\mid\mathbf{X}_{i})=\frac{\widehat{f}_{-i}(\mathbf{X}_{i}\mid\boldsymbol{h}_{i}) \pi (\boldsymbol{h}_{i})}{\int_{\mathbb{T}_{d}}\widehat{f}_{-i}(\mathbf{X}_{i}\mid\boldsymbol{h}_{i}) \pi (\boldsymbol{h}_{i})d\boldsymbol{h}_{i}}.\label{bayesrule}
	\end{eqnarray}
The Bayesian estimator  $\widetilde{\boldsymbol{h}}_{i}$ of $\boldsymbol{h}_{i}$ is obtained through the usual quadratic loss function as
	\begin{equation}\label{equ6}
		\widetilde{\mathbf{\boldsymbol{h}}}_{i}=\mathbb{E}\left(\boldsymbol{h}_{i}\mid\mathbf{X}_{i})=(\mathbb{E}(h_{i1}\mid\mathbf{X}_{i}),\ldots,\mathbb{E}(h_{id}\mid\mathbf{X}_{i})\right)^{\top}. 
	\end{equation}
	
	In what follows, we assume that each component $h_{i \ell}=h_{i \ell}(n)$, $\ell=1,\ldots,d$, of $\boldsymbol{h}_{i}$ has the univariate inverse-gamma prior  $\mathcal{IG}(\alpha,\beta_{\ell})$ distribution with the same shape parameters $\alpha > 0$ and, eventually, different scale parameters  $\beta_{\ell} > 0$ such that $\boldsymbol{\beta}=(\beta_1,\ldots,\beta_d)^\top$. We recall that the pdf of $\mathcal{IG}(\alpha,\beta_\ell)$ with $\alpha,\beta_\ell>0$ is determined by 
	\begin{equation}\label{prior}
		IG_{\alpha,\beta_\ell}(u)=\frac{\beta_\ell^{\alpha}}{\Gamma(\alpha)} u^{-\alpha-1} \exp(-\beta_\ell/u)\mathbf{1}_{(0, \infty)}(u), \;\;\ell=1,\ldots,d,
	\end{equation} 
where $\Gamma(\cdot)$ denotes the usual gamma function. 
	The mean and the variance of the prior distribution \eqref{prior} for each component $h_{i\ell}$ of the vector $\boldsymbol{h}_i$ are identified by $\beta_\ell/(\alpha -1)$ for $\alpha>1$, and $\beta_\ell^2/(\alpha-1)^2(\alpha-2)$ for $\alpha>2$, respectively.
	Note that for a fixed $\beta_\ell>0$, $\ell=1,\ldots,d$, and if $\alpha \rightarrow \infty$, then the distribution of the bandwidth vector $\boldsymbol{h}_{i}$ is concentrated around the null vector $\mathbf{0}$; easy to see for $d=1$. 
	
	The consideration of the inverse-gamma prior distribution is legitimate since, for instance, we explicitely get the corresponding posterior distribution in the following theorem.  
	
	\begin{theorem}\label{theo_Bayes}
		For a fixed $i \in \{1,2,\ldots,n\}$, consider each observation $\mathbf{X}_{i}=(X_{i1},\ldots,X_{id})^{\top}$ with its corresponding vector  $\boldsymbol{h}_{i}=(h_{i1},\ldots,h_{id})^{\top}$ of univariate bandwidths and defining the subsets  $\mathbb{I}_{ia_k}=\left\{k \in \{1,\ldots,d\};X_{ik} = a_k\right\}$, $\mathbb{I}_{ib_s}=\left\{s \in \{1,\ldots,d\}~;X_{is} = b_s\right\}$ and their complementary set $\mathbb{I}^{c}_{i}=\left\{\ell \in \{1,\ldots,d\}~;X_{i\ell} \in (a_\ell, b_\ell)\right\}$. Using the inverse-gamma prior $IG_{\alpha,\beta_{\ell}}$ of (\ref{prior}) for  each component $h_{i\ell}$ of $\boldsymbol{h}_{i}$ in the MEBK estimator (\ref{Adapbetaestimator}) with $\alpha>3/2$ and $\boldsymbol{\beta}=(\beta_1,\ldots,\beta_d)^\top\in(0,\infty)^d$ as $n\to\infty$, then:
		
		(i) the posterior density \eqref{bayesrule} has the behavior of a weighted sum of inverse-gamma
		\begin{eqnarray*}
			\pi(\boldsymbol{h}_{i}\mid\mathbf{X}_{i})&=&\frac{1}{D_{i}(\alpha,\boldsymbol{\beta})}\sum_{j=1,j\neq i}^{n} \left(\prod_{k \in \mathbb{I}_{ia_k}}[F_{jk}(\alpha,\beta_k)\,IG_{\alpha,E_{jk}(\beta_k)}(h_{ik})+H_{jk}(\alpha,\beta_k)\,IG_{\alpha+1,E_{jk}(\beta_k)}(h_{ik})]\right) \nonumber\\
			&&\quad\times \left(\prod_{\ell \in \mathbb{I}_{i}^c} [A_{ij\ell}(\alpha,\beta_\ell)\,IG_{\alpha+1/2,B_{ij\ell}(\alpha,\beta_\ell)}(h_{i\ell})+C_{ij\ell}(\alpha,\beta_\ell)\,IG_{\alpha-1/2,B_{ij\ell}(\alpha,\beta_\ell)}(h_{i\ell})]\right)\nonumber\\
			&&\quad\times \left(\prod_{k \in \mathbb{I}_{ib_s}} [J_{js}(\alpha,\beta_s)\,IG_{\alpha,G_{js}(\beta_s)}(h_{is})+K_{js}(\alpha,\beta_s)\,IG_{\alpha+1,G_{js}(\beta_k)}(h_{is})]\right)
		\end{eqnarray*}
with 
\begin{eqnarray*}
        A_{ij\ell}(\alpha,\beta_\ell)&=& \{2\pi(X_{i\ell}-a_{\ell})(b_{\ell}-X_{i\ell})\}^{-1/2}\beta^{\alpha}_{\ell}[B_{ij\ell}(\alpha,\beta_\ell)]^{-\alpha-1/2}\Gamma(\alpha+1/2),\\
        B_{ij\ell}(\alpha,\beta_\ell)&=&\beta_{\ell}-(b_{\ell}-a_{\ell})^{-1}\{(X_{i\ell}-a_{\ell})\log[(X_{j\ell}-a_{\ell})/(X_{i\ell}-a_{\ell})]\\
        &&+(b_{\ell}-X_{i\ell})\log[(b_{\ell}-X_{j\ell})/(b_{\ell}-X_{i\ell})]\},\\
		C_{ij\ell}(\alpha,\beta_\ell)&=& \{2\pi(X_{i\ell}-a_{\ell})(b_{\ell}-X_{i\ell})\}^{-1/2}\beta^{\alpha}_{\ell}[B_{ij\ell}(\alpha,\beta_\ell)]^{-\alpha+1/2}\Gamma(\alpha-1/2),\\
		E_{jk}(\beta_k)&=&\beta_k-\log[(b_k-X_{jk})/(b_k-a_k)],\\
		F_{jk}(\alpha,\beta_k)&=&\beta_{k}^{\alpha}(b_k-a_k)^{-1}\Gamma(\alpha)[E_{jk}(\beta_k)]^{-\alpha},\\
		G_{js}(\beta_s)&=&\beta_s-\log[(X_{js}-a_s)/(b_s-a_s)],\\
		H_{jk}(\alpha,\beta_k)&=&\beta_{k}^{\alpha}(b_k-a_k)^{-1}\Gamma(\alpha+1)E_{jk}^{-\alpha-1}(\beta_k),\\
				J_{js}(\alpha,\beta_s)&=&\beta_{s}^{\alpha}(b_s-a_s)^{-1}\Gamma(\alpha)G_{js}^{-\alpha}(\beta_s),\\
		K_{js}(\alpha,\beta_s)&=&\beta_{s}^{\alpha}(b_s-a_s)^{-1}\Gamma(\alpha+1)[G_{js}(\beta_s)]^{-\alpha-1} 
\end{eqnarray*}
and
\begin{eqnarray*}
		D_{i}(\alpha,\boldsymbol{\beta})&=&\sum_{j=1,j\neq i}^{n}  \left(\prod_{k \in \mathbb{I}_{ia_k}}\left\{F_{jk}(\alpha,\beta_k)+H_{jk}(\alpha,\beta_k)\right\}\right)
		\left(\prod_{\ell \in \mathbb{I}_{i}^c}\left\{ A_{ij\ell}(\alpha,\beta_\ell)+C_{ij\ell}(\alpha,\beta_\ell)\right\}\right)\\
		& &\quad\quad\times\left(\prod_{k \in \mathbb{I}_{ib_s}} \left\{ J_{js}(\alpha,\beta_s)+K_{js}(\alpha,\beta_s)\right\}\right);
\end{eqnarray*}
		
		(ii) under the quadratic loss function, the Bayesian estimator $\widetilde{\boldsymbol{h}}_{i}=\left(~\widetilde{h}_{i1},\ldots,\widetilde{h}_{id}\right)^{\top}$  of $\boldsymbol{h}_{i}$, introduced in ($\ref{equ6}$), is 
		\begin{align}\label{hbayes}
			\widetilde{h}_{im} &= \frac{1}{D_{i}(\alpha,\boldsymbol{\beta})}\sum_{j=1,j\neq i}^{n} \left(\prod_{k \in \mathbb{I}_{ia_k}}[F_{jk}(\alpha,\beta_k)+H_{jk}(\alpha,\beta_k)]\right) \nonumber\\
			&\quad \times \left(\prod_{\ell \in \mathbb{I}_{i}^c} [A_{ij\ell}(\alpha,\beta_\ell)+C_{ij\ell}(\alpha,\beta_\ell)]\right)\left(\prod_{k \in \mathbb{I}_{ib_s}} [J_{js}(\alpha,\beta_s)+K_{js}(\alpha,\beta_s)]\right)\nonumber\\
			&\quad \times\left\{\left(\frac{F_{jm}(\alpha,\beta_m)}{\alpha-1}+\frac{H_{jm}(\alpha,\beta_m)}{\alpha}\right)\frac{E_{jm}(\beta_m)}{F_{jm}(\alpha,\beta_m)+H_{jm}(\alpha,\beta_m)}\mathbf{1}_{\{a_m\}}(X_{im})\right. \nonumber \\ 
			&\left.\quad\quad + \left(\frac{A_{ijm}(\alpha,\beta_m)}{\alpha-1/2}+\frac{C_{ijm}(\alpha,\beta_m)}{\alpha-3/2}\right)\frac{B_{ijm}(\alpha,\beta_m)}{A_{ijm}(\alpha,\beta_m)+C_{ijm}(\alpha,\beta_m)}\mathbf{1}_{(a_m,b_m)}(X_{im})\right. \nonumber \\ 
			&\left.\quad\quad + \left(\frac{J_{jm}(\alpha,\beta_m)}{\alpha-1}+\frac{K_{jm}(\alpha,\beta_m)}{\alpha}\right)\frac{G_{ijm}(\beta_m)}{J_{jm}(\alpha,\beta_m)+K_{jm}(\alpha,\beta_m)}\mathbf{1}_{\{b_m\}}(X_{im})\right\},
		\end{align}
		for $\;m=1,2,\ldots,d$, with the previous notations of $A_{ij\ell}(\alpha,\beta_\ell)$, $B_{ijm}(\alpha,\beta_m)$, $C_{ij\ell}(\alpha,\beta_\ell)$, $E_{jk}(\beta_k)$, $F_{jm}(\alpha,\beta_m)$, $G_{ijm}(\beta_m)$, $H_{jk}(\alpha,\beta_k)$, $J_{jm}(\alpha,\beta_m)$ and $K_{js}(\alpha,\beta_s)$.
	\end{theorem}

\begin{remark}\label{Rem:BayesParam}
The implementation of the Bayesian adaptive approach with inverse-gamma priors requires the choice of parameters $\alpha$ and $\boldsymbol{\beta}=(\beta_1,\ldots,\beta_d)^\top$. We recall that the optimal bandwidth that minimizes the mean integrated squared
error is of $\mathcal{O}\left(n^{-2/5}\right)$ order for the univariate extended-beta kernel; go back to \cite{LK17} for further details. Thus, departing from the first and second moments of the prior \eqref{prior}, we need to consider $\alpha=\alpha_n=n^{2/5}>2$ and $\beta_\ell = 1$ for all $\ell=1,2,\ldots,d$ to ensure the convergence of the variable bandwidths to zero and also in practice; see also Tables \ref{Sensitivity_analysis_n} and \ref{Sensitivity_analysis} of sensitivity analysis. . 
\end{remark}	

Although these previous choices in Remark \ref{Rem:BayesParam} do not generate necessarily the best smoothing quality; it would be possible to modify them slightly in practice in order to enhance the smoothing. See below Figure \ref{Fig:EstC} depicting a sensitivity analysis and certain numerical illustrations; see equally \cite{SK20} for gamma kernels.

When the compact support $\mathbb{T}_d$ of the unknown $d$-variate pdf $f$ will be estimated from its random sample $\mathbf{X}_{1},\ldots,\mathbf{X}_{n}$, we can consider the simple and intuitive nonparametric support estimator $\widehat{\mathbb{T}}_d$ which has been introduced by  \cite{DevroyeW1980} as unions of the closed Euclidian balls centered at $\mathbf{X}_{i}=(X_{i1},\ldots,X_{id})^\top$, $i=1,\ldots,n$, with common radius $b_n$ which is a global smoothing parameter. For a review along with some properties, we refer the reader, for instance, to \cite{Baillo2000} and \cite{Biau2009}. 

\begin{remark}\label{Rem:SuppotEstim}
Within the framework of the adaptive vector bandwidths $\boldsymbol{h}_i(n)=\left(h_{i1}(n),\ldots,h_{id}(n)\right)^\top$, we suggest to use in practice the following $\widehat{\mathbb{T}}_d$ defined as 
\begin{equation*}
\widehat{\mathbb{T}}_d:=\bigcup_{i=1}^n\left(\times_{j=1}^d[X_{ij}\pm h_{ij}(n)]\right)
= \times_{j=1}^d[X_{(1)j}-h_{(1)j}(n),X_{(n)j}+h_{(n)j}(n)],
\end{equation*}
where $X_{(1)j}:=\min\{X_{ij};i=1,\ldots,n\}$, $X_{(n)j}:=\max\{X_{ij};i=1,\ldots,n\}$ and $h_{(k)j}(n)$ represents the corresponding adaptive bandwidth of $X_{(k)j}$ for $k=1,n$ and $j=1,\ldots,d$.
\end{remark}	

\section{Simulation results}
\label{sec_simulation}

This section is dedicated to the numerical findings of simulation studies designed to assess the performance of the suggested approach so as to capture the exact shape of unknown densities. Notably, we focus on Bayesian adaptive bandwidth selection for MEBK estimators. The computations were carried out on a PC equipped with a 2.30 GHz processor using the \textsf{R} software \citep{R20}.

These simulation studies aim to achieve three primary objectives related to unbiased cross-validation (UCV) method. Firstly, we explore the capability of our elaborated method, as described in  \eqref{Cn_f} and \eqref{hbayes} with Theorem \ref{theo_Bayes}, to generate accurate estimates, as denoted in Equation \eqref{Adapbetaestimator}, of unknown true densities on the  given domain $\mathbb{T}_{d}=\times_{j=1}^{d}[a_{j}, b_{j}]$ or on an estimated one $\widehat{\mathbb{T}}_d$, for $d \geq 1$ and $a_{j} < b_{j}<\infty$; see also Remark \ref{Rem:SuppotEstim}. Secondly, we assess the sensitivity of the proposed method, in univariate case, concerning varying sample sizes $n$ and prior parameters $\alpha$ and $\beta$. Lastly, we compare the computational time needed for both UCV and our Bayesian adaptive procedure.

In fact, we obtain the optimal bandwidth vector $\boldsymbol{h}_{ucv}$ of $\boldsymbol{h}$ in Equation \eqref{betaestimator} through UCV as follows:
\[
\boldsymbol{h}_{ucv} = \arg \min_{\boldsymbol{h} \in (0,\infty)^d} UCV(\boldsymbol{h}),
\]
where
\[
UCV(\boldsymbol{h}) = \int_{\times_{j=1}^{d}[a_{j}, b_{j}]}{\left\{\widehat{f}_n(\boldsymbol{x})\right\}^{2}}d\boldsymbol{x} - \frac{2}{n}\sum_{i=1}^{n}{\widehat{f}_{n,\boldsymbol{h},-i}(\mathbf{X}_i)},
\]
with
\[
{\widehat{f}_{n,\boldsymbol{h},-i}(\mathbf{X}_i)}=\frac{1}{n-1}\sum_{\ell=1,\ell\neq i}^n\prod_{j=1}^d EB_{X_{ij},h_{j}}(X_{\ell j}),
\]
being computed as $\widehat{f}_n(\mathbf{X}_i)$ of Equation \eqref{betaestimator} by excluding the observation $\mathbf{X}_i$. The efficiency of the smoothers shall be examined through the empirical estimate $\widehat{ISE}$ of the integrated squared errors (ISE):
\[
ISE := \int_{\times_{j=1}^{d}[a_{j}, b_{j}]}\left\lbrace\widetilde{f}_n(\boldsymbol{x})- f(\boldsymbol{x})\right\rbrace^{2} d\boldsymbol{x},
\]
where $\widetilde{f}_n$ stands for the normalized kernel estimators \eqref{Cn_f} and \eqref{betaestimator}   with UCV method or the variable multivariate one $\widetilde{f}_n$ of Equations \eqref{Cn_f}  and \eqref{Adapbetaestimator} for our adaptive Bayesian method. All $\widehat{ISE}$ values are computed with the number of replications $N=100$. The results of these different simulations are reported in Tables \ref{Sensitivity_analysis_n}-\ref{stat_desc_univ} and Figures \ref{Fig:EstC}-\ref{Fig:SimuABC} for uni-/multivariate cases $(d=1,2,3 \mbox{ and } 5)$.

\subsection{Univariate case and sensitivity analysis}

At this stage, we consider four scenarios denoted A, B, C and D to simulate bounded datasets with respect to the support of univariate extended-beta kernel (i.e. $\mathbb{S}_{x,h}= [a,b]=\mathbb{T}$ with $a<b$). These scenarios have also been selected to compare the performances of smoothers \eqref{Cn_f} with \eqref{betaestimator} and \eqref{Adapbetaestimator} with regard to convex, unimodal, multimodal or U-shape  distributions.   
\begin{itemize}
	\item Scenario A is generated by an asymmetric bell-shaped beta distribution
	$$
	f_{A}(x)=\frac{\left(x-1\right)^{4}\left(5-x\right)^{0}}{1024\mathscr{B}(5, 1)} \mathbf{1}_{\left[1, 5\right]}(x);
	$$ 
	
	\item Scenario B is a PERT (or extended-beta) distribution
	$$
	f_{B}(x)=\frac{\left(x-1\right)\left(5-x\right)^{3}}{1024\mathscr{B}(2, 4)} \mathbf{1}_{\left[1, 5\right]}(x);
	$$
	\item Scenario C is a mixture of two PERT distributions 
	$$
	f_{C}(x)=\left(\frac{3}{5}\frac{\left(x-1\right)^{6/11}\left(5-x\right)^{38/11}}{\mathscr{B}(17/11, 49/11)\left(11\right)^{4}} +\frac{2}{5}\frac{\left(x-1\right)^{36/11}\left(5-x\right)^{8/11}}{\mathscr{B}(47/11, 19/11)\left(11\right)^{4}}\right) \mathbf{1}_{\left[-2, 9\right]}(x);
	$$
	\item Scenario D derives from  logit-normal distribution with a U-shape
	$$f_{D}(x) = \frac{1}{3 \sqrt{2 \pi}} \frac{1}{ x (1-x)} \exp\left(-\left[\log\left(\frac{x}{1-x}\right) -0.25\right]^2 \middle/ 18\right)\mathbf{1}_{(0,1)}(x).$$	
\end{itemize}

\begin{figure}[!ht]
	\centering
	\subfloat[]{%
		\resizebox*{9cm}{!}{\includegraphics[width=8.cm, height=8cm]{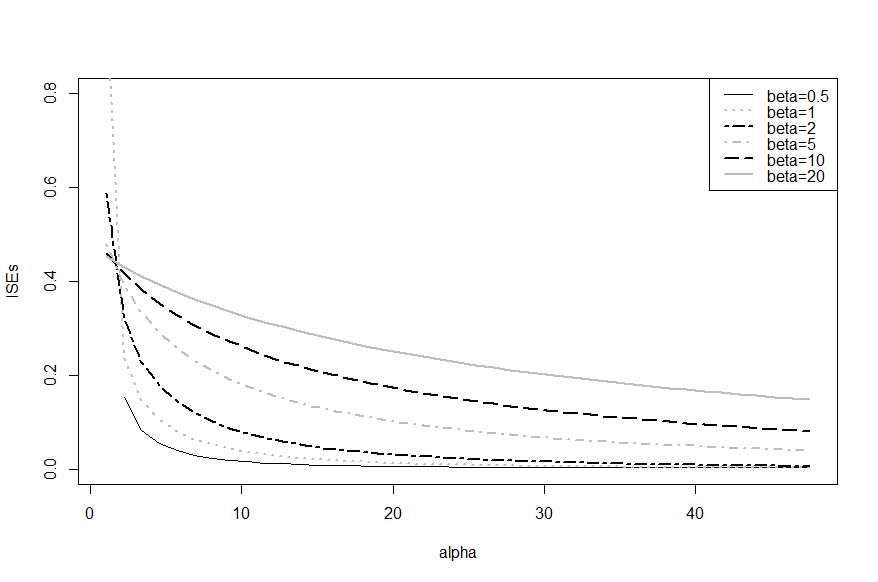}}}	
	\caption{Plots of the ISE using extended-beta  smoother \eqref{betaestimator} in Scenario B with $n = 500$ and 		for different values of $\alpha$ and $\beta_\ell = \beta$ of prior distribution \eqref{prior}.}\label{Fig:EstC}
\end{figure}

First, a sensitivity analysis of the performances of the univariate extended-beta smoother \eqref{betaestimator} is conducted using the prior hyperparameters $\alpha=\alpha_{n}=n^{2/5}$ and $\beta_j>0$, $\in{1,\ldots,d}$. Proceeding in the same way as the previously stated procedure, these choices are expected to ensure the convergence of the adaptive Bayesian tuning parameters to zero. Figure \ref{Fig:EstC} sheds light on the behavior of $\beta \mapsto ISE(\beta)$ on  Scenario B for  six values $\beta \in\{0.5,1,2,5,10,20\}$ with sample size $n=500$. We observe that $\alpha=500^{2/5}=12.0112$ with $\beta <1$ lead nearly to the minimum values of $\widehat{ISE}$. Additionally, increasing values of $\beta$ imply lower convergences to zero of the estimated ISEs. Furthermore, Table \ref{Sensitivity_analysis_n} reveals that  the ISE values become smaller when the sample sizes increase for a given $\alpha \in\{ 2,5,10,15, 20, 30\}$, $\beta \in\{0.1,0.25,0.5,1,2,5\}$ and sample sizes $n=100, 300$ and 500. As plotted in Figure \ref{Fig:EstC}, the lower values of $\beta <1$ display a very satisfactory smoothing quality with $\alpha=n^{2/5}=6.3095$, $8.3255$ and $12.0112$ for $n=100, 300$ and 500, respectively. These choices are not necessarily the optimum ones to obtain the best smoothing quality. Finally, Table \ref{Sensitivity_analysis} reveals all these behaviours for all scenarios A, B, C, D, $\alpha \in\{ 2,5,10,15, 20, 30\}$, $\beta \in\{0.25,0.5,1,2,5\}$ and only for a larger sample size $n=500$. Moreover, smoothing quality consistently increases for uni-/bi-modal scenarios compared to convex distributions, and even excels further when contrasted with U-shape ones.

\begin{table}[!htbp]
	\caption{ Expected values ($\times 10^3$) and their standard deviations in parentheses of $\widehat{ISE}$ with $N=100$ replications  for different values of $\alpha$ and $\beta_{\ell}=\beta$ in Scenario  B.}
	{	\begin{tabular}{crrrrrrr}
			\hline
			$\beta_{\ell}=\beta$	&$\alpha$&\multicolumn{1}{c}{$100$}&\multicolumn{1}{c}{$200$ }&\multicolumn{1}{c}{$500$ }    \\\hline		
			\multirow{6}{*}{$0.1$}& \multirow{1}{*}{2} &18.7287 (6.3474)
			&18.3059 (3.8405)&17.2210 (2.7223)\\
			&\multirow{1}{*}{5} &9.1042 (4.8952)&4.2784 (2.5375)&3.7654 (1.9544)\\
			&\multirow{1}{*}{10}&14.4691 (8.5864)& 4.8034 (2.5289)&2.8789 (1.4144) \\
			&\multirow{1}{*}{15}&18.6310 (8.3888)&4.6785 (3.1343)&4.0605 (2.0708)\\
			&\multirow{1}{*}{20}&17.2893 (7.8134)&6.1571 (2.7281)&4.0990 (2.0571)\\
			&\multirow{1}{*}{30}&26.8312 (12.1096)&8.4850 (2.8171)&5.4464 (2.3582)\\
			\hline		
			
			\multirow{6}{*}{$0.25$}& \multirow{1}{*}{2} &43.1767 (4.3814)& 43.1625 (2.7282)&43.4670 (2.6271)\\
			&\multirow{1}{*}{5} &12.1927 (6.2350)&10.5988 (3.4377)&9.2283 (2.4320)\\
			&\multirow{1}{*}{10}&9.1671 (5.2336)& 5.1064 (3.9126)&3.8825 (1.9345) \\
			&\multirow{1}{*}{15}&11.8874 (7.6368)&4.1319 (2.2354)&3.4151 (1.9849)\\
			&\multirow{1}{*}{20}&13.1828 (7.7800)&4.3724 (2.4961)&2.9817 (1.8935)\\
			&\multirow{1}{*}{30}&15.3643 (7.1915)&5.5363 (2.7498)&3.1472 (1.7428)\\
			\hline		
			
			\multirow{6}{*}{$0.5$}& \multirow{1}{*}{2} &74.97054 (4.4021)&74.7534 (1.9998)&74.3326 (1.8412)\\
			&\multirow{1}{*}{5} &24.3335 (6.5485)&22.2138 (3.7373) &22.0404 (2.4209)\\
			&\multirow{1}{*}{10}&11.6635 (6.8541)&  9.3185 (3.0857)&8.4889 (2.2937) \\
			&\multirow{1}{*}{15}&9.2364 (5.6802)&6.1882 (3.2737)&4.4179 (1.6962)\\
			&\multirow{1}{*}{20}&9.7162 (7.1223)&5.7172 (3.5299)&3.6523 (1.8512)\\
			&\multirow{1}{*}{30}&10.1732 (5.8679)&4.2829 (2.3619)&4.1718 (2.0531)\\
			\hline
			
			\multirow{6}{*}{1}& \multirow{1}{*}{2}  &105.1241 (3.5493)&105.4775 (2.0350)&105.2314 (1.5433) \\	
			&\multirow{1}{*}{5}&43.0005 (4.9704)&41.8417 (2.6120 )&41.9609 (2.6125)\\
			&\multirow{1}{*}{10}&21.8739 (5.9657)&20.3022 (3.2174)&19.7279 (2.6202)\\
			&\multirow{1}{*}{15}&15.5373 (7.9574)&12.5755 (3.7075)&11.7732 (2.7833)\\
			&\multirow{1}{*}{20}&11.4649 (5.4483)&9.1578 (3.6119)&9.1415 (2.9671)\\
			&\multirow{1}{*}{30}&7.8978 (5.4276)&5.9386 (2.5145)&5.1303 (2.1371)\\
			\hline 
			\multirow{6}{*}{2}& \multirow{1}{*}{2}  &126.1075 (2.2575)&126.3647 (1.2414)&126.1581 (0.9926)\\
			&\multirow{1}{*}{5}&69.0168 (4.5467)&69.3718 (2.5288)&68.6149 (2.0435)\\
			&\multirow{1}{*}{10}&39.7711 (4.5098)&38.8178 (2.7311)&38.0123 (1.9684)\\
			&\multirow{1}{*}{15}&28.4291 (5.9826)&25.1541 (2.9731)&26.5472 (2.6282)\\
			&\multirow{1}{*}{20}&22.1367 (7.4000)&19.4300 (3.5791)&18.6882 (2.7722)\\
			&\multirow{1}{*}{30}&14.8064 (6.0704)&12.4760 (4.0745)&11.7934 (2.7008)\\
			\hline 
			\multirow{6}{*}{5}
			& \multirow{1}{*}{2}&139.1179 (0.9074)&139.1025 (0.5417)& 139.1103 (0.3637) \\ 
			&\multirow{1}{*}{5}&104.9637 (3.5356)&104.9425 (1.5118)&104.7899 (1.3307)\\
			&\multirow{1}{*}{10}&74.2584 (5.0845)&73.8912 (2.1682)&73.2654 (1.6401)\\
			&\multirow{1}{*}{15}&55.7195 (4.8940)&55.9576 (2.1805)&55.7695 (2.0331)\\
			&\multirow{1}{*}{20}&44.9449 (4.8001)&44.7349 (2.9943)&44.8518 (2.2091)\\	
			&\multirow{1}{*}{30}&32.0037 (5.5659)&31.6434 (2.4360)&31.8495 (2.6716)\\	
			\hline
	\end{tabular}}
	\label{Sensitivity_analysis_n}
\end{table} 

\begin{table}[!htbp]
	\caption{ Expected values ($\times 10^3$) and their standard deviations in parentheses of $\widehat{ISE}$ with $n=500$, $N=100$ replications  for different values of $\alpha$ and $\beta_{\ell}=\beta$ in Scenarios A, B, C and D.}
	{	\begin{tabular}{crrrrrrr}
			\hline
			$\beta_{\ell}=\beta$	&$\alpha$&\multicolumn{1}{c}{$\widehat{ISE}_{A}$}&\multicolumn{1}{c}{$\widehat{ISE}_{B}$ }&\multicolumn{1}{c}{$\widehat{ISE}_{C}$ }&\multicolumn{1}{c}{$\widehat{ISE}_{D}$ }    \\\hline		
			
			\multirow{6}{*}{$0.25$}& \multirow{1}{*}{2} &63.8438 (7.2056)& 43.4670 (2.6271)&16.3635 (0.3782)&383.5863 (22.2413)\\
			&\multirow{1}{*}{5} &9.0693 (4.2419)&9.2283 (2.4320)&5.8430 (0.6635)&162.8053 (19.0522)\\
			&\multirow{1}{*}{10}&3.7929 (2.6865)& 3.8825 (1.9345)&2.5855 (0.9735)&119.3822 (15.2558) \\
			&\multirow{1}{*}{15}&3.7638 (2.9230)&3.4151 (1.9849)&1.7976 (0.7535) &130.4549 (14.7381)\\
			&\multirow{1}{*}{20}&4.1106 (3.2985)&2.9817 (1.8935)&1.3575 (0.5927)&125.3927 (13.6744)\\
			&\multirow{1}{*}{30}&5.1270 (3.1515)&3.1472 (1.7428)&1.3943 (0.5857)&158.7471 (28.6446)\\
			\hline		
			
			\multirow{6}{*}{$0.5$}& \multirow{1}{*}{2} &152.7892 (6.8013)&74.3326 (1.8412)&20.3976 (0.2804)&514.8053 (11.2818)\\
			&\multirow{1}{*}{5} &22.7191 (5.3033)&22.0404 (2.4209)&10.6514 (0.6107)&259.3906 (25.4461)\\
			&\multirow{1}{*}{10}&9.5054 (4.7427)&  8.4889 (2.2937)&5.1302 (0.8876)&164.6618 (17.3598) \\
			&\multirow{1}{*}{15}&4.7907 (3.5728)&4.4179 (1.6962)&3.3448 (0.7940)&133.3951 (18.4147)\\
			&\multirow{1}{*}{20}&4.4400 (3.7643)&3.6523 (1.8512)&2.3837 (0.7789)&121.0451 (13.6062)\\
			&\multirow{1}{*}{30}&4.0428 (2.4452)&4.1718 (2.0531)&1.5003 (0.6352)&117.2256 (12.0876)\\
			\hline
			
			\multirow{6}{*}{1}& \multirow{1}{*}{2}  &261.2879 (5.3160)&105.2314 (1.5433)&22.3939 (0.2219)& 584.8077 (3.6770)\\	
			&\multirow{1}{*}{5}&62.9280 (6.5670)&41.9609 (2.6125)&15.9701 (0.4509)&356.9220 (21.3331)\\
			&\multirow{1}{*}{10}&21.7644 (4.6890)&19.7279 (2.6202)&9.8481 (0.5760)&240.7989 (21.0377)\\
			&\multirow{1}{*}{15}&10.6315 (4.2699)&11.7732 (2.7833)&6.8407 (0.6573)&186.2314 (23.3280)\\
			&\multirow{1}{*}{20}&7.1220 (3.9561)&9.1415 (2.9671)&5.1353 (0.9249)&160.6313 (19.1270)\\
			&\multirow{1}{*}{30}&5.1178 (3.0613)&5.1303 (2.1371)&3.3004 (0.7536)&135.4001 (18.6294)\\
			\hline 
			\multirow{6}{*}{2}& \multirow{1}{*}{2}  &346.9498 (2.5905)&126.1581 (0.9926)&23.6747 (0.1220)&610.7817 (1.6115)\\
			&\multirow{1}{*}{5}&132.8912 (6.1537)&68.6149 (2.0435)&19.6494 (0.2765)&463.8578 (13.3935)\\
			&\multirow{1}{*}{10}&56.3239 (7.7844)&38.0123 (1.9684)&15.0934 (0.3721)&332.8535 (21.5032)\\
			&\multirow{1}{*}{15}&29.4171 (6.9545)&26.5472 (2.6282)&11.9898 (0.5531)&271.1553 (23.7918)\\
			&\multirow{1}{*}{20}&18.8908 (5.1843)&18.6882 (2.7722)&9.6315 (0.5236)&228.5061 (22.3149)\\
			&\multirow{1}{*}{30}&10.4811 (3.6480)&11.7934 (2.7008)&6.5363 (0.7434)&183.3435 (18.9414)\\
			\hline 
			\multirow{6}{*}{5}
			& \multirow{1}{*}{2}&406.5433 (1.1392)&139.1103 (0.3637)&24.5520 (10.0483)& 621.7521 (0.6549)  \\ 
			&\multirow{1}{*}{5}&258.0829 (4.8576)&104.7899 (1.3307)&22.2954 (0.1869)&567.3611 (17.2918)\\
			&\multirow{1}{*}{10}&145.1572 (8.1281)&73.2654 (1.6401)&19.9091 (0.2749)&516.9667 (12.2717)\\
			&\multirow{1}{*}{15}&92.4672 (5.5989)&55.7695 (2.0331)&18.0590 (0.3550)&448.6337 (21.6033)\\
			&\multirow{1}{*}{20}&65.6836 (7.4113)&44.8518 (2.2091)&16.2298 (0.3685)&392.0456 (18.8470)\\	
			&\multirow{1}{*}{30}&36.1508 (6.3353)&31.8495 (2.6716)&13.4353 (0.5259)&333.2493 (24.2946)\\	
			\hline
	\end{tabular}}
	\label{Sensitivity_analysis}
\end{table} 

Next, Table \ref{tab:computtimes} illustrates the time required for computing both bandwidth selection methods across single replications of sample sizes $n=10,25,50,100,200, 500$ and 1000 for the target function A. The findings distinctly highlight the superiority of the Bayesian adaptive approach over the UCV method. Furthermore, the contrast in Central Processing Unit (CPU) times grows more pronounced with rising sample sizes.

\begin{table}[!htbp]
	\centering
	\caption{Typical CPU times (in seconds) of both UCV and Bayesian adaptive bandwidth 	selections \eqref{hbayes}  using	one replication of Scenario B.} 
	{	\begin{tabular}{rrrrrr}
			\hline
			$n\;$  &\multicolumn{1}{c}{$t_{UCV}$ }&\multicolumn{1}{c}{$t_{Bayes}$}\\\hline
			
			10   &11.97  &0.00\\
			25  & 30.06 & 0.03  \\
			50   & 59.71  & 0.03 \\
			100  &122.22 &0.14\\
			200     & 261.85 &0.38 \\
			500    & 710.67 &  3.49\\
			1000    & 1619.75  &10.93 \\
			\hline
	\end{tabular}}
	\label{tab:computtimes}
\end{table}

Figure \ref{Fig:SimuABC} depicts the true density as well as  the smoothing densities for both UCV and Bayesian adaptive bandwidth selectors with extended-beta kernel estimators related to the considered models, and for only one replication. Basically, the performances are quite similar exhibiting the same difficulties in capturing the excat shape for convex distributions (e.g., Part A1–A2 of Figure \ref{Fig:SimuABC}).
\begin{figure}[!htbp]
	\vspace*{-2.5cm}
	\centering
	\subfloat[(A1)]{%
		\includegraphics[width=7.cm, height=5.45cm]{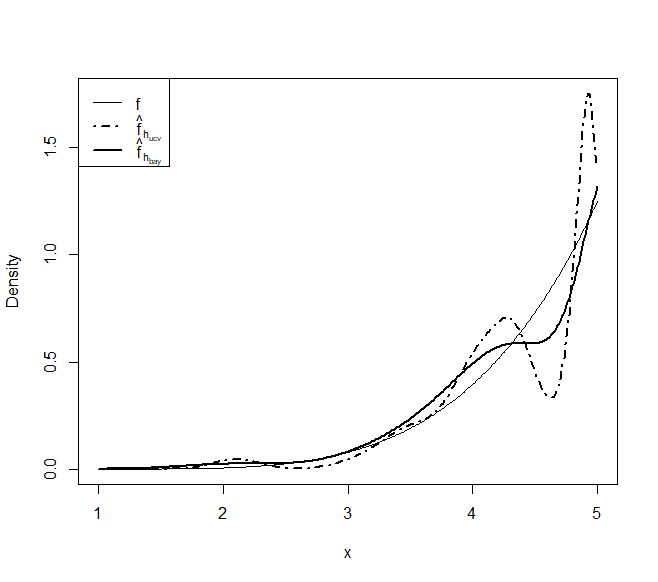}}
	\subfloat[(A2)]{%
		\includegraphics[width=7.cm, height=5.45cm]{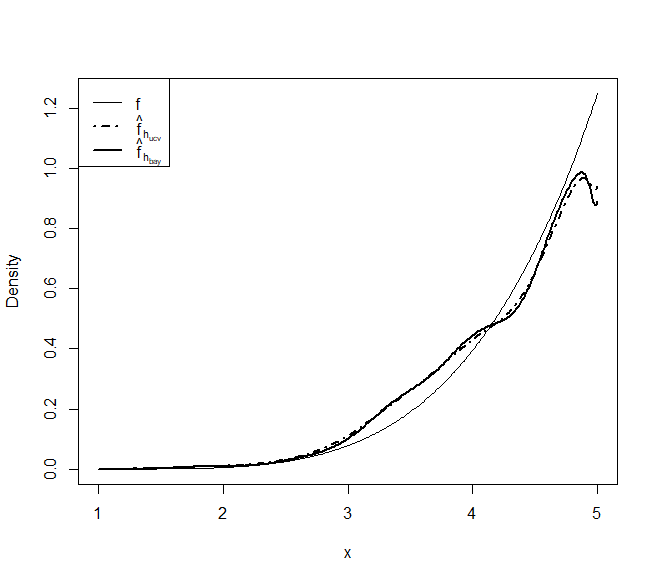}}\\
	\vspace*{-.475cm}	
	\subfloat[(B1)]{%
		\includegraphics[width=7.cm, height=5.45cm]{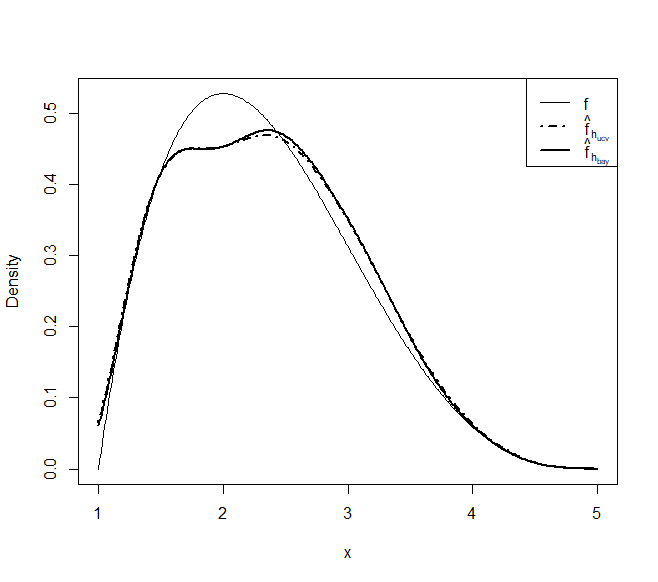}}
	\subfloat[(B2)]{%
		\includegraphics[width=7.cm, height=5.45cm]{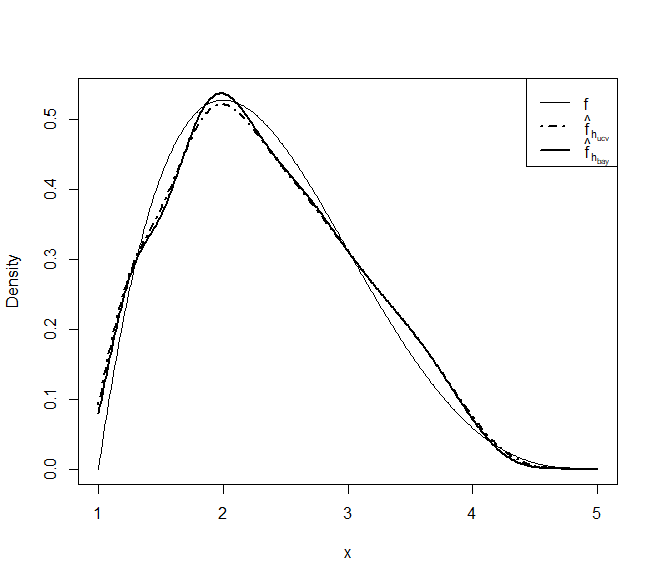}}\\
	\vspace*{-.475cm}
	\subfloat[(C1)]{%
		\includegraphics[width=7.cm, height=5.45cm]{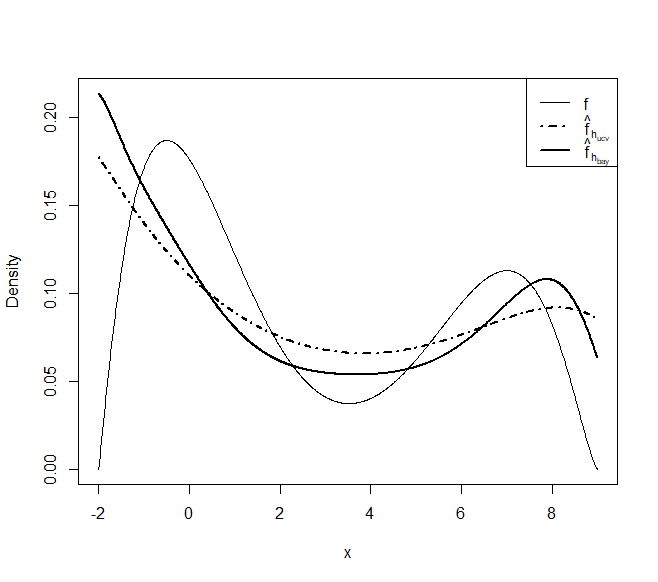}}
	\subfloat[(C2)]{%
		\includegraphics[width=7.cm, height=5.45cm]{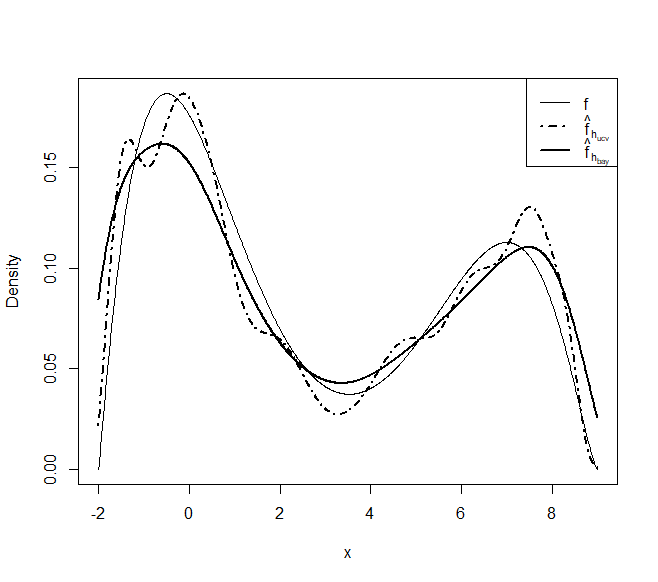}}\\
	\vspace*{-.475cm}
	\subfloat[(D1)]{%
		\includegraphics[width=7.cm, height=5.45cm]{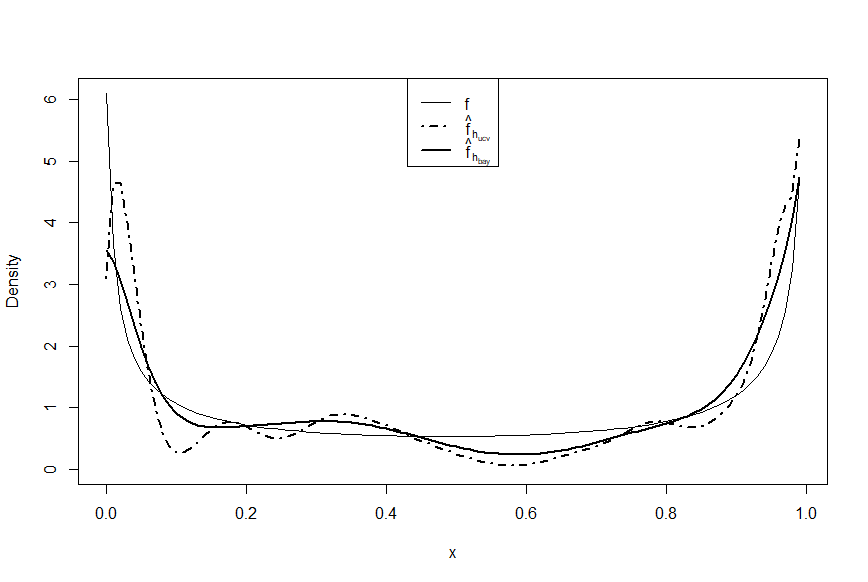}}
	\subfloat[(D2)]{%
		\includegraphics[width=7.cm, height=5.45cm]{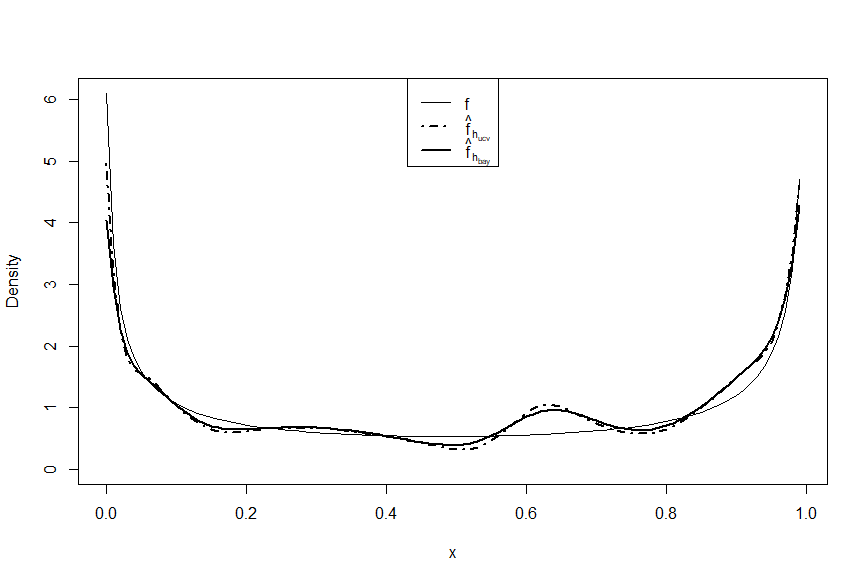}}
	\caption{True pdf and its corresponding smoothings using  extended-beta kernels with UCV and Bayesian adaptive bandwidths for Scenarios A, B, C and D with $n=50$ (left); $n=200$ (right). }\label{Fig:SimuABC}
\end{figure}

Table \ref{ErrGamma} displays some expected values of $\widehat{ISE}$ with respect to the four Scenarios A, B, C,  and D, and according to the sample sizes $n = 20$, 50, 100, 200, 500. Thus, we infer the following behaviors for both UCV and Bayesian tuning methods. As expected, the smoothings get better as the sample sizes increase. The Bayesian approach is clearly better than the UCV one for all samples size and all considered models. The difference of smoothing quality is smaller for larger sample sizes (i.e. $n=500$ and 1000). Both methods seem to have difficulties in terms of smoothing enough convex distributions and even more  U-shape ones; e.g., Parts A and D of Table \ref{ErrGamma}. 

\begin{table}[!htbp]
	\centering
	\caption{ Expected values ($\times 10^3$) and their standard deviations in parentheses of $\widehat{ISE}$ with $N=100$ replications  using	global UCV and Bayesian adaptive bandwidths (with prior parameters $\alpha=n^{2/5}>2$ and $\beta_{1}=1/4$) for all univariate Scenarios A, B, C and D.} 
	{	\begin{tabular}{rrrrccc}
			\hline
			&$n$&\multicolumn{1}{c}{$\widehat{ISE}_{UCV}$}&\multicolumn{1}{c}{$\widehat{ISE}_{Bayes}$ }   \\\hline		
			
			\multirow{7}{*}{A}& 	\multirow{1}{*}{10} &	115.8009 (119.8057)&	91.5578 (65.4830)\\
			&	\multirow{1}{*}{20}  &82.0050 (86.2389)&78.2823 (86.2389)\\ 
			
			&\multirow{1}{*}{50} &28.8678 (26.6640)  &16.6236 (13.8403)&\\
			&\multirow{1}{*}{100}&14.4373 (14.2570)& 11.3620 (10.7994) \\
			&\multirow{1}{*}{200}&10.7450 (9.9227) & 8.5159 (9.1898)\\
			&\multirow{1}{*}{500}&4.1750 (2.5537)& 4.3024 (3.6539)\\
			&\multirow{1}{*}{1000}&3.0077 (2.2489)& 2.5830 (2.0727)\\
			\hline 
			\multirow{7}{*}{B}& \multirow{1}{*}{10}  &131.5619 (142.1452)&51.1909 (51.2114) \\	
			&\multirow{1}{*}{20}&57.9147 (47.2916)&35.3433 (31.5520)\\
			&\multirow{1}{*}{50}&29.9359 (24.0753)&15.2518 (13.3836)\\
			&\multirow{1}{*}{100}&17.0892 (12.3043)&14.1387 (12.5297)\\
			&\multirow{1}{*}{200}&9.1671 (7.0321)&6.7505 (6.1550)\\
			&\multirow{1}{*}{500}&4.5435(3.8951)&4.3132(4.2758)\\
			&\multirow{1}{*}{1000}&1.9028(1.1876)&1.9655(1.2394)\\
			\hline 
			\multirow{7}{*}{C}
			
			& \multirow{1}{*}{10}&24.24109 (15.9713)& 16.93086 (5.3865) \\ 			
			&\multirow{1}{*}{20}&18.0597 (11.4405)&12.2576 (4.3113) \\
			&\multirow{1}{*}{50}&9.3314 (5.5185)&7.8745 (3.1088)\\
			&\multirow{1}{*}{100}&6.4595 (2.6487)&5.4703 (2.2149)\\
			&\multirow{1}{*}{200}&2.9873 (1.2543)&3.6521 (1.3319)\\
			&\multirow{1}{*}{500}&1.3891 (0.6393)&2.0440 (0.7081)\\
			&\multirow{1}{*}{1000}&0.8519 (0.3368)&1.2272 (0.4151)\\
			\hline
			
			\multirow{7}{*}{D}& \multirow{1}{*}{10}  & 883.8449 (915.7151)&419.7151 (166.3445)\\			
			&\multirow{1}{*}{20}&524.7636 (391.8950)&296.8862 (125.4970) \\
			&\multirow{1}{*}{50}&355.0219 (174.9401)&213.1530 (67.6651)\\
			&\multirow{1}{*}{100}&264.1031 (117.9244)&182.3334 (55.0430)\\
			&\multirow{1}{*}{200}&188.1498 (55.2339)&146.2614 (36.7259)\\
			&\multirow{1}{*}{500}&155.5031 (26.2496)&118.4205 (14.6518)\\
			&\multirow{1}{*}{1000}&	141.0610 (14.7649)&	113.1771 (6.8750)\\
			\hline
	\end{tabular}}
	\label{ErrGamma}
\end{table} 

\subsection{Multivariate cases}

In order to investigate simulations in dimensions beyond or equal to $d \geq 2$, we analyze five test pdfs corresponding to Scenarios E, F, G, H, and I. These pdfs are examined for $d = 2$, 3, and 5, incorporating both independent and correlated margins, respectively.

\begin{itemize}	
	\item 	Scenario E is generated using the bivariate PERT distribution 
	\begin{equation*}
		f_{E}(x_{1}, x_{2}) = 	\frac{\left(x_1-1\right)\left(5-x_1\right)^{3}}{1024\mathscr{B}(2, 4)}\times \frac{\left(x_2-1\right)\left(5-x_2\right)^{3}}{1024\mathscr{B}(2, 4)} \mathbf{1}_{\left[1, 5\right]^2}(x_1,x_2);
	\end{equation*}
	\item Scenario F is a mixture of three bivariate normal pdfs  with a  negative correlation structure \citep[from][]{DH05}
	$$
	f_{F}\left(x_1,x_2\right)=\frac{4}{11} \mathcal{N}_2\left(\mu_2, \Sigma_2\right)+\frac{3}{11} \mathcal{N}_2\left(\mu_3, \Sigma_3\right)+\frac{4}{11} \mathcal{N}_2\left(\mu_4, \Sigma_4\right),
	$$
	with $\mu_2=(5,9)^{\top}, \mu_3=(7,7)^{\top}, \mu_4=(9,5)^{\top}, \Sigma_2=\left(\begin{array}{rr}1 & 0 \\ 0 & 1\end{array}\right), \Sigma_3=\left(\begin{array}{rr}0.80 & -0.72 \\ -0.72 & 0.80\end{array}\right)$ and $\Sigma_4=\left(\begin{array}{rr}1 & 0 \\ 0 & 1\end{array}\right)$;
	\item Scenario G is the truncated bivariate normal distribution denoted $\mathcal{TN}_{2}\left(\mu_5, \Sigma_5\right)$ on truncated domain $[-1,5]\times[-1,5]$
	$$
	f_{G}\left(x_1,x_2\right)=\mathcal{TN}_{2}\left(\mu_5, \Sigma_5\right),
	$$  with mean vector ${\mu}_5=(0,0)^{T}$ and covariance matrix
	${\Sigma}_5=\left(
	\begin{array}{cc}
		1 & 0.8 \\
		0.8 & 1 \\
	\end{array}
	\right)$;
	\item Scenario H  is a trivariate mixture of two PERT distributions 
	\begin{eqnarray*}
		f_{H}(x_1,x_2,x_3)&=&\prod_{\ell=1}^3\left(\frac{3}{5}\frac{\left(x_\ell-1\right)^{6/11}\left(5-x_\ell\right)^{38/11}}{\mathscr{B}(17/11, 49/11)\left(11\right)^{4}} +\frac{2}{5}\frac{\left(x_\ell-1\right)^{36/11}\left(5-x_\ell\right)^{8/11}}{\mathscr{B}(47/11, 19/11)\left(11\right)^{4}}\right) \\
		&&\times\; \mathbf{1}_{\left[-2, 9\right]^3}(x_1,x_2,x_3);
	\end{eqnarray*}
	\item Scenario I is a $5$-variate  extended-beta distribution 
	$$
	f_{I}(x_1,x_2,x_3,x_4,x_5)=\prod_{\ell=1}^5\left(\frac{\left(x_\ell-1\right)\left(5-x_\ell\right)^{3}}{1024\mathscr{B}(2, 4)}\right) \mathbf{1}_{\left[1, 5\right]^5}(x_1,x_2,x_3,x_4,x_5).
	$$
\end{itemize}

Table \ref{ErrGamma_1} portrays the  smoothing investigation within the multivariate framework (for dimensions $d = 2, 3$, and $5$), focusing on densities E, F, G, H, and I. We explore sample sizes of $n = 20, 50, 100, 200, 500,$ and $1000$ for each model. Additionally, for dimension $d = 5$, we introduce sample sizes $n \in {2000, 5000}$, as $n = 1000$ can be considered in this context moderate or small. Across all sample sizes, ISE values display a very acceptable degree of smoothing.

\begin{table}[!ht]
	\centering
	\caption{ Expected values ($\times 10^3$) and their standard deviations in parentheses of $\widehat{ISE}$ with $N=100$ replications  using	global UCV and Bayesian adaptive bandwidths (with prior parameters $\alpha=n^{2/5}>2$ and $\beta_{1}=1/4$) for all multivariate Scenarios E, F, G, H and I.} 
	{	\begin{tabular}{rrrrccc}
			\hline
			&$n$&&\multicolumn{1}{c}{$\widehat{ISE}_{Bayes}$ }   \\\hline		
			
			\multirow{5}{*}{E}
			&	\multirow{1}{*}{20}  & &26.2165 (10.9673)\\ 
			
			&\multirow{1}{*}{50} &  &16.8885 (7.9130)&\\
			&\multirow{1}{*}{100}&& 11.2012 (4.2758) \\
			&\multirow{1}{*}{200}& & 7.4936 (3.3838)\\
			&\multirow{1}{*}{500}&& 4.1682 (1.2423)\\
			\hline 
			\multirow{5}{*}{F}
			&\multirow{1}{*}{20}&&17.9456 (2.4143)\\
			&\multirow{1}{*}{50}&&14.1387 (1.2529)\\
			
			&\multirow{1}{*}{100}&&12.0115 (1.0830)\\
			&\multirow{1}{*}{200}&&9.0897(1.3408)\\
			&\multirow{1}{*}{500}&&7.4302(0.8877)\\
			\hline 
			\multirow{5}{*}{G}
			
			& \multirow{1}{*}{20}&& 23.7323 (9.2013) \\ 			
			&\multirow{1}{*}{50}&&14.6309 (4.6107) \\
			&\multirow{1}{*}{100}&&9.6835 (3.7420)\\
			&\multirow{1}{*}{200}&&7.3814 (2.7526)\\
			&\multirow{1}{*}{500}&&3.9949 (1.2094)\\
			\hline
			
			\multirow{6}{*}{H}
			&\multirow{1}{*}{20}&&2.9588 (0.4699) \\
			&\multirow{1}{*}{50}&&2.8549 (0.3186)\\
			&\multirow{1}{*}{100}&&2.6787 (0.2648)\\
			&\multirow{1}{*}{200}&&2.5187 (0.2028)\\
			&\multirow{1}{*}{500}&&2.4163 (0.1262)\\
			&\multirow{1}{*}{1000}&&	2.1639 (0.1210)\\
			\hline
			\multirow{8}{*}{I}
			&\multirow{1}{*}{20}&&6.4316 (1.3573)\\
			&\multirow{1}{*}{50}&&5.7605 (1.3839)\\
			&\multirow{1}{*}{100}&&5.0563 (1.0231)\\
			&\multirow{1}{*}{200}&&4.5133 (0.7663)\\
			&\multirow{1}{*}{500}&&3.7025 (0.3025)\\
			&\multirow{1}{*}{1000}&&2.0125 (0.1256)	&\\
			&\multirow{1}{*}{2000}&&	1.2398 (0.0502)&\\
			&\multirow{1}{*}{5000}&	&1.0873 (0.0372)&\\
			\hline
	\end{tabular}}
	\label{ErrGamma_1}
\end{table} 

\section{Illustrative applications}
\label{sec_application}

The performance of the nonparametric method introduced in this study is assessed through three distinct examples featuring (new for) univariate, (usual for) bivariate, and (original for) trivariate real datasets. Specifically, we invest multivariate extended-beta kernel estimators \eqref{betaestimator} with Bayesian adaptive bandwidths \eqref{hbayes} or UCV to smooth the joint distributions. The goodness-of-fit of these estimators is sometimes contrasted with those using a gamma kernel with Bayesian adaptive bandwidth, as reported by \cite{SomeEtAl2024}, as well as a Gaussian smoother with plug-in tuning parameters, as clarified in \cite{Duong2007} and \cite{DuongEtAl23}. Unless explicitly stated otherwise, the prior model parameters for bandwidths are consistently set to $\alpha=n^{2/5}$ and $\beta_\ell=0.1$ for all $\ell=1,\ldots,d$ for the extended-beta kernel, and $\alpha=n^{2/5}$ and $\beta_\ell=1$ for all $\ell=1,\ldots,d$ for the gamma kernel estimation. We recall that the previous opted for parameters are not necessarily the optimal ones; see, e.g., resuts of the sentivity analysis summarized in Figure \ref{Fig:EstC} and Table \ref{Sensitivity_analysis} for the extended-beta kernel and, Figure 1 and Table 1 of \cite{SK20} for the gamma kernel.

In addition to graphical comparisons of smoothings, a numerical procedure, following the approach of \cite{FiliSAngui11}, is applied for our three illustrative examples. Initially, subsets of size $m_n$ are sampled for each sample size $n$: $m_n \in 15, 30, 50, 70$ for the first univariate dataset with $n=82$, $m_n \in 20, 30, 40, 50$ for the second univariate dataset with $n=86$, $m_n \in 100, 150, 200, 250$ for the bivariate dataset with $n=272$, and $m_n \in 100, 250, 400, 500$ for the trivariate dataset with $n=590$. Subsequently, the remaining data are used to compute the average log-likelihood. This process is repeated $100$ times for each $m_n$; see also  \cite{SomeEtAl2024}.

The three univariate illustrative examples first concern total cholesterol levels for $n=82$ age groups of $1986$ individuals sampled during a period of 12 months in the year 2020 at the ``Laboratoire National de Biologie Clinique et Sant\'e Publique (LBCSP)'' in Bangui (Central African Republic). This is to obtain the overall trend, regardless of age. Tables \ref{univ_data} and \ref{stat_desc_univ} provide a descriptive summary of data, with skewness (CS), among other indicators. These nonnegative data are left-skewed within the range $[1,2]$ and suggest the use of asymmetric kernels and in particular the extended-beta smoother. Figure \ref{Fig:cholesterol} depicts the histogram and the smoothed distributions on data support $\mathbb{T}_1=[1,2]$ and an estimated extension of $\mathbb{T}_1$ (i.e. $\widehat{\mathbb{T}}_1=[0.95,2.05]$), using extended-beta kernels, namely the solid line for Bayesian adaptive selector of tuning parameter and dashed line for cross-validation one; consult also Remark \ref{Rem:SuppotEstim}. An estimated support from the range of data proves to be extremely useful as cholesterol data may have a slightly bigger range in other samples or regions. We record nearly identical performances with regard to both methods (UCV and Bayesian adaptive), with a significant impact observed on the smoothing at the edges owing to extension $\widehat{\mathbb{T}}_1=[0.95,2.05]$. Thus, the estimated support can be crucial in terms of capturing the exact shape of the density. These previous statements are put into evidence by the numerical results of the cross-validated average log-likelihood method displayed in Table \ref{log_cholesterol}.
\begin{table}[!htbp]
	\centering
	\caption{Average data of cholesterol level ($X_i$) per age  ($a$) in \texttt{ml/dl} (or \texttt{g/l}) for $n=82$ age groups of $1986$ individuals sampled during the year 2020 in Bangui (Central African Republic).} 
	\begin{tabular}{cc|cc|cc|cc}
		\hline $a$& $X_i$& $a$& $X_i$ & $a$&$X_i$ & $a$&$X_i$ \\\hline
		2 & 1.40 & 30 & 1.65 & 51 & 1.57 & 72 & 1.47 \\
		3 & 2.00 & 31 & 1.80 & 52 & 1.73 & 73 & 1.69 \\
		4 & 1.80 & 32 & 1.76 & 53 & 1.75 & 74 & 1.57 \\
		8 & 1.70 & 33 & 1.32 & 54 & 1.70 & 75 & 1.48 \\
		10 & 1.50 & 34 & 1.61 & 55 & 1.71 & 76 & 1.42 \\
		12 & 1.58 & 35 & 1.68 & 56 & 1,62 & 77 & 1.70 \\
		13 & 1.90 & 36 & 1.59 & 57 & 1.55 & 78 & 1.,41 \\
		14 & 1.03 & 37 & 1.74 & 58 & 1.65 & 79 & 2.00 \\
		15 & 1.00 & 38 & 1.79 & 59 & 1.78 & 80 & 1.79 \\
		16 & 1.90 & 39 & 1.66 & 60 & 1.64 & 81 & 1.60 \\
		17 & 1.14 & 40 & 1.73 & 61 & 1.72 & 82 & 2.00 \\
		18 & 1.43 & 41 & 1.59 & 62 & 1.56 & 83 & 1.77 \\
		21 & 1.16 & 42 & 1,64 & 63 & 1.61 & 84 & 1.72 \\
		22 & 1.50 & 43 & 1.64 & 63 & 1.61 & 84 & 1.72 \\
		23 & 1.40 & 44 & 1.57 & 65 & 1.70 & 86 & 1.83 \\
		24 & 1.62 & 45 & 1.62 & 66 & 1.57 & 87 & 1.50 \\
		25 & 1.53 & 46 & 1.69 & 67 & 1.57 & 89 & 1.70 \\
		26 & 1.51 & 47 & 1.65 & 68 & 1.67 & 90 & 1.73 \\
		27 & 1.58 & 48 & 1.63 & 69 & 1.54 & 96 & 1.80 \\
		28 & 1.75 & 49 & 1.55 & 70 & 1.61 & & \\
		29 & 1.56 & 50 & 1.67 & 71 & 1.62 & & \\
		\hline
	\end{tabular}	
	\label{univ_data}
\end{table} 
\begin{table}[!htbp]
	\begin{center}
		\caption{Summary of the analyzed average cholesterol ($X_i$) data of Table \ref{univ_data}} \label{stat_desc_univ}
		\begin{tabular}{lrrrrrrr}
			\hline
			Data set &\multicolumn{1}{c}{$n$}  &\multicolumn{1}{c}{max.}&\multicolumn{1}{c}{min.} &\multicolumn{1}{c}{median} &\multicolumn{1}{c}{mean}&\multicolumn{1}{c}{SD}&\multicolumn{1}{c}{CS}\\\hline
			$X_i$	& 82 & 2.000  & 1.000 &1.635  &1.621  &   0.181&$-0.976$\\
			
			\hline
		\end{tabular}
	\end{center}
\end{table}
\begin{table}[!htbp]
	\caption{Mean average log-likelihood and its standard errors (in parentheses) for cholesterol data of Table \ref{univ_data} and based on $100$ replications using univariate extended-beta smoother with UCV and Bayesian adaptive bandwidths for $\alpha=n^{2/5}$ and $\beta_1=0.5$.} 
	\begin{center}
		{		\begin{tabular}{rrrrrr}
				\hline
				&$m_n$&\multicolumn{1}{c}{UCV}&\multicolumn{1}{c}{Bayesian adaptive}  \\\hline
				& \multirow{1}{*}{15} &0.3538 (0.0567) &0.3012 (0.0331) \\	 
				&\multirow{1}{*}{30}& 0.2861 (0.1382) &0.2832 (0.0616) \\ 
				&\multirow{1}{*}{50}&0.1761 (0.2607)& 0.2449 (0.0863)\\
				&\multirow{1}{*}{70}&0.0967 (0.3475)&	0.1951 (0.1068)  \\ 		
				
				\hline
		\end{tabular}}
		\label{log_cholesterol}
	\end{center}
\end{table}

\begin{figure}[!htbp]
	\vspace*{-1.cm}
	\centering
	\subfloat[(a)]{%
		\includegraphics[width=8.cm, height=7cm]{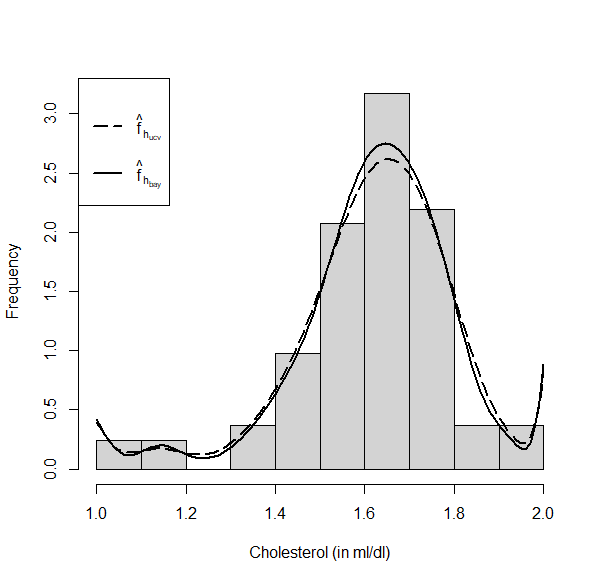}}
	\subfloat[(b)]{%
		\includegraphics[width=8.cm, height=7cm]{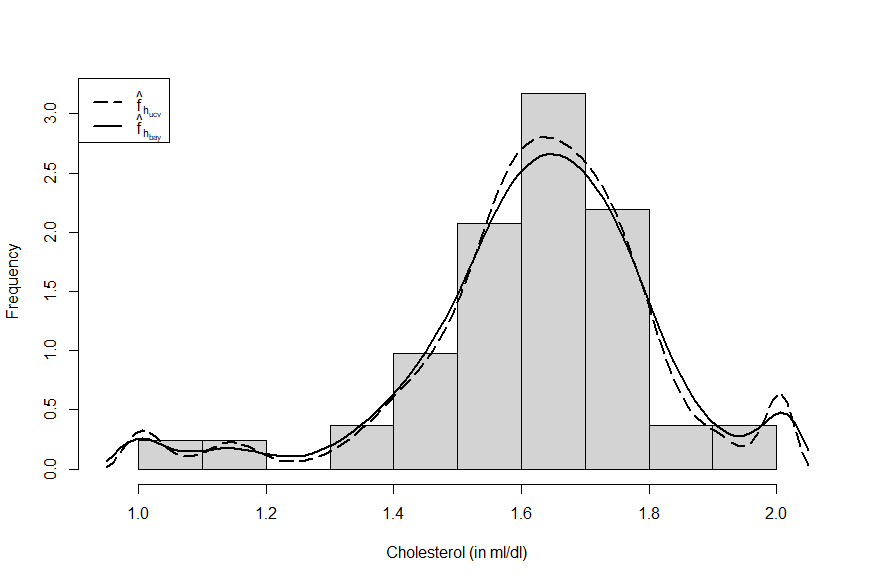}}\\
	\caption{Histogram with its corresponding smoothings of the cholesterol data of Table \ref{univ_data} using 	univariate extended-beta kernels with both UCV and Bayes selectors of bandwidths for $\alpha=n^{2/5}$ and $\beta=0.5$ on: (a) $[a_1,b_1]=[1,2]$ and, (b) $[a_1,b_1]=[0.95,2.05]$. }\label{Fig:cholesterol}
\end{figure}

The other univariate examples refer to students' marks   in $[0,20]$ for 'Introduction to Practical Work' ($X_1$), 'Written and oral expression techniques' ($X_2$) and 'Scientific English' ($X_3$) for the first year ($n=590$) of the Mathematics-Physics-Chemistry-Computer Science program at Universit\'e Thomas Sankara; refer back to Tables \ref{stat_desc_notes} and \ref{tri_data1} for some descriptive summary. Both variables $X_1$ and $X_3$ are left-skewed while $X_3$ is right-skewed, suggesting the use of extended-beta kernels that can have the same behavior; see Figure \ref{Fig:betakern}. Additionally, two extended-beta smoothers are considered to estimate both the marginal densities; also, we use the Gaussian one with the best tuning parameter (plug-in) available for this dataset with the \texttt{ks} package  of \cite{Duong2007}. As previously asserted, both extended-beta kernel estimators are comparable and better than the Gaussian smoother; see Figure \ref{Fig:marks} and also Table \ref{stat_desc_notes} of average log-likelihood.
\begin{table}[!htbp]
	\begin{center}
		\caption{Summary of the analyzed marks data of Table \ref{tri_data1}} \label{stat_desc_notes}
		\begin{tabular}{lrrrrrrr}
			\hline
			Data set &\multicolumn{1}{c}{$n$}  &\multicolumn{1}{c}{max.}&\multicolumn{1}{c}{min.} &\multicolumn{1}{c}{median} &\multicolumn{1}{c}{mean}&\multicolumn{1}{c}{CS}\\\hline
			$X_1$	& 570 & 20.00  & 0.00&11.35  &7.61  &   $-0.0603$&\\
			$X_2$ & 570 & 20.00 &0.00& 5.00 &5.495  & 0.2972 \\
			$X_3$&570 &	20.00 &0.00 &8.90&7.147& $-0.0310$ \\		
			\hline
		\end{tabular}
	\end{center}
\end{table}

\begin{figure}[!htbp]
	\vspace*{-1.cm}
	\centering
	\subfloat[$X_1$]{%
		\includegraphics[width=8.cm, height=7cm]{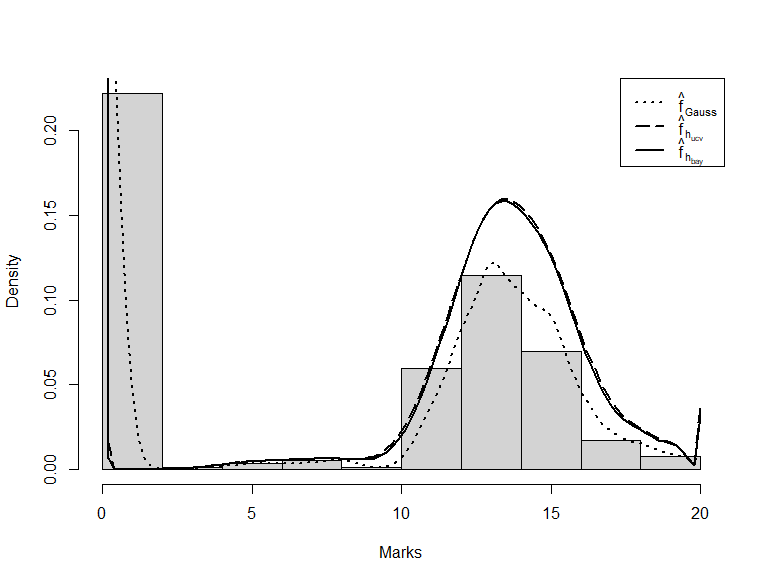}}
	\subfloat[$X_2$]{\includegraphics[width=8.cm, height=7cm]{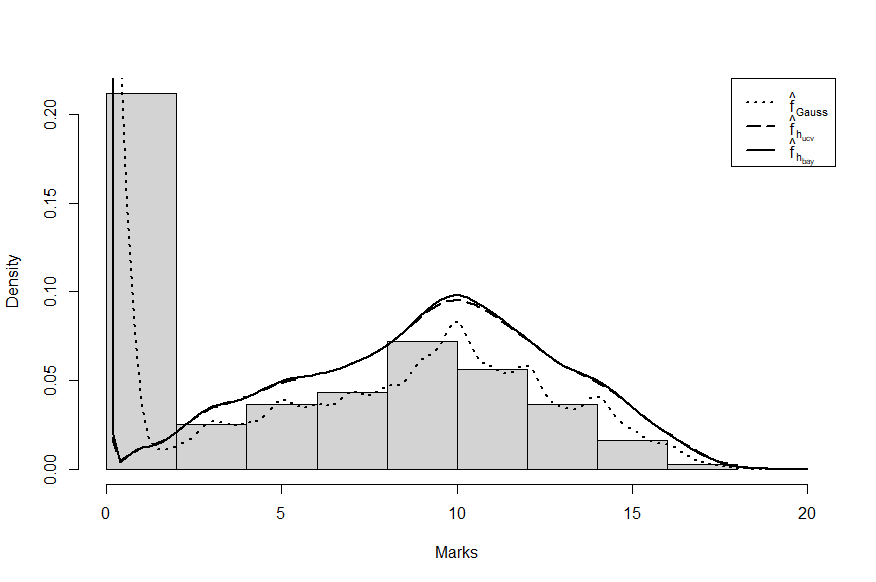}}
	\caption{Histogram with its corresponding smoothings of the two first marks data from Table \ref{tri_data1} ($X_1$ and $X_2$) using univariate Gaussian kernel with plug-in selector of bandwidth and univariate extended-beta kernels with both UCV and Bayes selectors of bandwidths for $\alpha=n^{2/5}$ and $\beta=0.1$ on $[a_1,b_1]=[0,20]$.}\label{Fig:marks}
\end{figure}
\begin{table}[!htbp]
	\caption{Mean average log-likelihood and its standard errors (in parentheses) for marks data exhibited in Table \ref{tri_data1} and based on 100 replications using univariate extended-beta kernels with both UCV and Bayes selectors of bandwidths for $\alpha=n^{2/5}$ and $\beta=0.1$ and Gaussian kernel estimator with plug-in selector of bandwidth matrix.} 
	\begin{center}
		{		\begin{tabular}{rrrrrr}
				\hline
				&$m_n$&\multicolumn{1}{c}{Extended-beta$_{UCV}$}&\multicolumn{1}{c}{Extended-beta$_{Bayes}$ }&\multicolumn{1}{c}{Gaussian$_{Plug-in}$ }&  \\\hline
				\multirow{4}{*}{$X_1$}	& \multirow{1}{*}{100} &  $-0.9140$ (0.0380) &$-0.8701$ (0.0387) &$-1.9992$ (0.0227)&	   \\
				&\multirow{1}{*}{250}&$-0.9259$ (0.0527) & $-0.9122$ (0.0531) & $-2.0438$ (0.0290)  &   \\ 
				&\multirow{1}{*}{400}&   $-0.9285$ (0.1061) &$-0.9779$ (0.1050)& $-2.1243$ (0.0627)  \\
				&\multirow{1}{*}{500}&$-1.0780$ (0.1919) &$-1.2242$ (0.1803)& $-2.3075$ (0.0805)& \\ 		
				\hline
				\multirow{4}{*}{$X_2$}	& \multirow{1}{*}{100}&$-1.3906$ (0.0404)& $-1.3588$ (0.0416)& $-2.3014$ (0.0262) &   \\
				&\multirow{1}{*}{250}&  $-1.3841$ (0.0696) &$-1.3781$ (0.0704) &$-2.3388$ (0.0429)&   \\ 
				
				&\multirow{1}{*}{400}&$-1.3820$ (0.1235)& $-1.4340$ (0.1210) &$-2.4128$ (0.0644) &\\ 
				
				&\multirow{1}{*}{500}	& $-1.4338$ (0.2137)& $-1.5977$ (0.1967) &$-2.5624$ (0.1042) & \\
				\hline
					
		\end{tabular}}
		\label{log_tri1_data}
	\end{center}
\end{table}

The bivariate illustration is derived from the Old Faithful geyser data already discussed in literature, see, e.g., \cite{FiliSAngui11} and recently \cite{SomeEtAl2024},  and available in the {\it Datasets Library} of the \textsf{R} software \cite{R20}. Data concern $n=272$ measurements of the eruption for the Old Faithful geyser in Yellowstone National Park, Wyoming, USA. Both covariates represent, in minutes, the {\it waiting time} between eruptions and the {\it duration} of the eruption of these nonnegative real data with mean vector and covariance matrix estimated as $\widehat{\mu}=(3.48,70.89)^{\top}$ and  $\widehat{\Sigma}=\left(\begin{array}{rr}1.30 & 13.97 \\ 13.97 & 184.82\end{array}\right)$, respectively. 

Figure~\ref{fig_Old_faithful_data} reports contour and surface plots of the smoothed distribution using gamma and extended-beta kernel with Bayesian bandwidth selections for the Old Faithful geyser dataset. The points designate the scatter plots and the solid lines indicate the contour plot estimates. The smoothing with gamma and extended-beta kernels are quite similar and point out  the bimodality of the Old Faithful geyser data. In contrast, the extended-beta smoother seems to be the best with respect to Table~\ref{log_Old_Geyser} of average log-likelihood.
\begin{table}[!htbp]
	\caption{Mean average log-likelihood and its standard errors (in parentheses) for  Old Faithful Geyser data based on 100 replications using multiple gamma and MEBK with Bayesian adaptive bandwidths ($\alpha=n^{2/5}$ and $\beta_1=\beta_2=1$)  and ($\alpha=n^{2/5}$ and $\beta_1=\beta_2=1/5$), respectively.} 
	\begin{center}
		{		\begin{tabular}{rrrrrr}
				\hline
				&$m_n$&\multicolumn{1}{c}{Gamma}&\multicolumn{1}{c}{Extended-beta }  \\\hline
				& \multirow{1}{*}{100} &$-792.41$ (4.95)&	$ -744.42$ (6.72)  \\	 
				&\multirow{1}{*}{150}&$-570.01$ (5.42)&$-532.49$ (6.45) \\ 
				&\multirow{1}{*}{200}&$-329.50$ (5.69) &  $-317.17$ (4.42) \\
				&\multirow{1}{*}{250}&$-111.28$ (2.47)& $-100.61$ (2.63) \\ 		
				
				\hline
		\end{tabular}}
		\label{log_Old_Geyser}
	\end{center}
\end{table}

\begin{figure}[!htbp]
	\subfloat[(Contour gamma)]{%
		\includegraphics[width=8.cm, height=7cm]{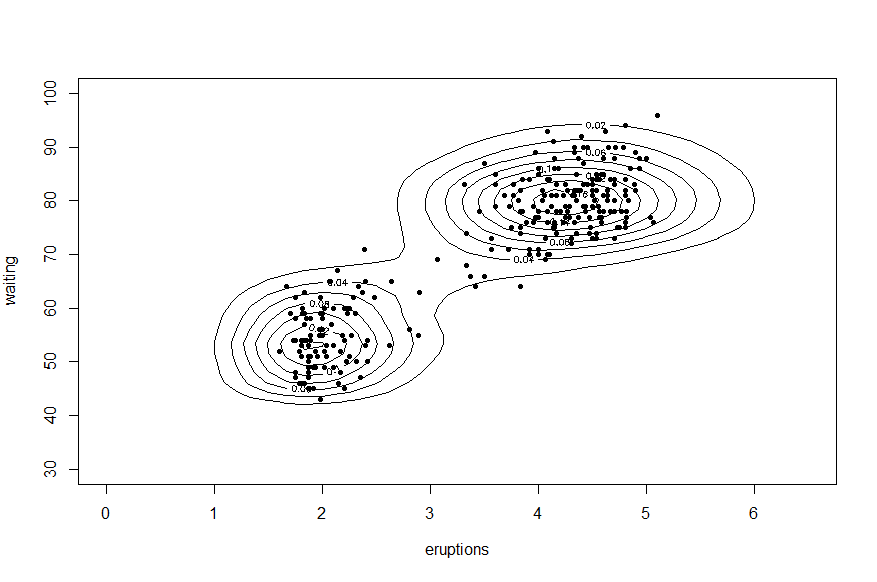}}
	\subfloat[(Surface gamma)]{%
		\includegraphics[width=8.cm, height=7cm]{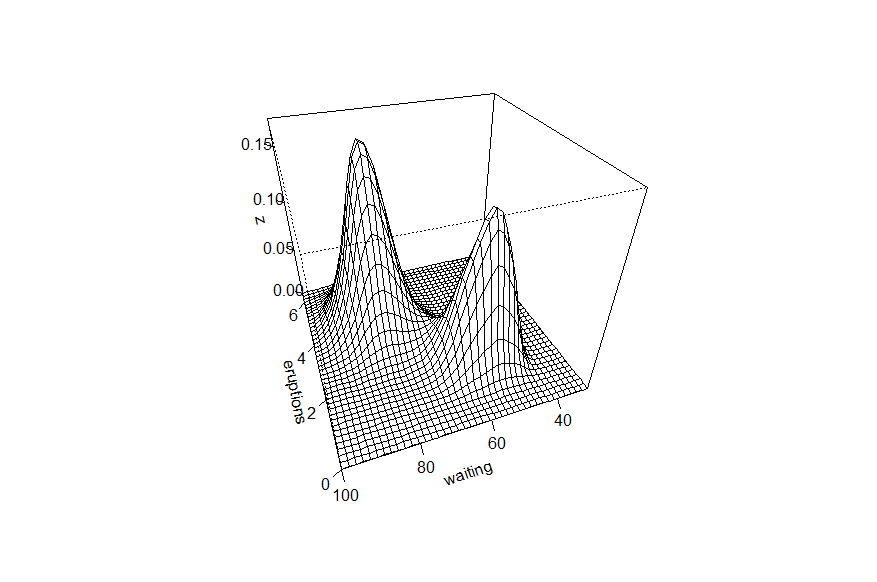}}\\
	\subfloat[(Contour extended-beta)]{%
		\includegraphics[width=8.cm, height=7cm]{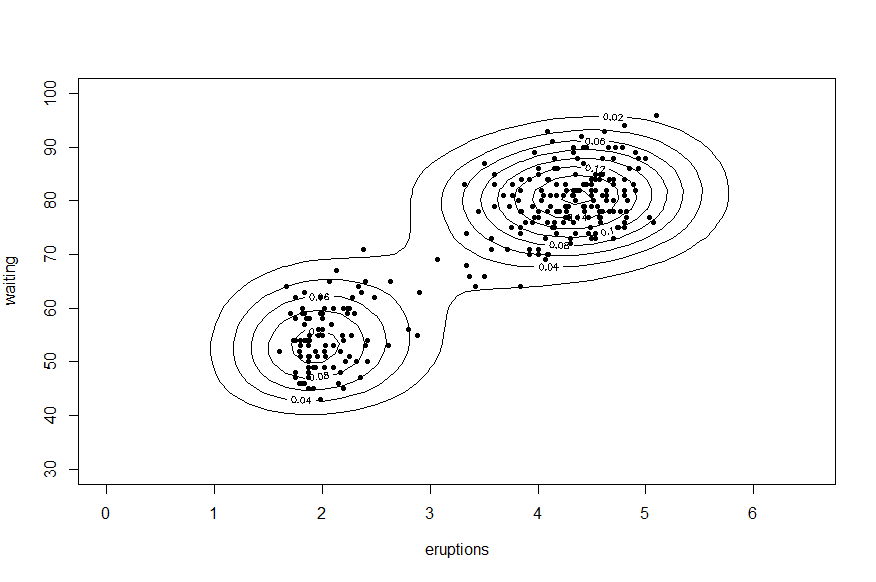}}
	\subfloat[(Surface extended-beta)]{%
		\includegraphics[width=8.cm, height=7cm]{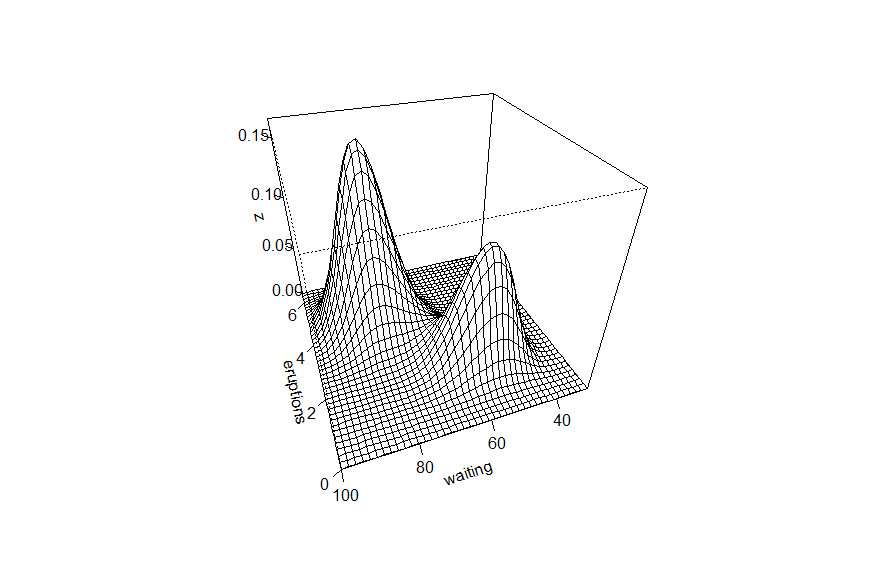}}
	\caption{Contour (left) and surface (right) plots of smoothed distribution from the Old Faithful geyser dataset according to Bayes selectors of bandwidths vector $\boldsymbol{h}$ with multiple gamma kernels ($\alpha=n^{2/5}$ and $\beta_1=\beta_2=1$) and MEBK ($\alpha=n^{2/5}$ and $\beta_1=\beta_2=1/5$).}
	\label{fig_Old_faithful_data}	
\end{figure}

\begin{table}[!htbp]
	\caption{Mean average log-likelihood and its standard errors (in parentheses) for marks data illustrated in Table \ref{tri_data1} and grounded on 100 replications using MEBK estimator with Bayes selectors of bandwidths for $\alpha=n^{2/5}$ and $\beta=0.1$, and multivariate Gaussian kernel estimator with plug-in selector of bandwidth matrix.} 
	\begin{center}
		{		\begin{tabular}{rrrrrr}
				\hline
				&$m_n$&\multicolumn{1}{c}{Extended-beta$_{Bayes}$ }&\multicolumn{1}{c}{Gaussian$_{Plug-in}$ }&  \\\hline			
				\multirow{4}{*}{$ALL$}	& \multirow{1}{*}{100}&$-4.3397$ (0.0810) &$-4.0658$ (0.1119)   \\
				&\multirow{1}{*}{250} & $-4.1458$ (0.1551) &$-4.2512$ (0.1674)   &  \\ 
				
				&\multirow{1}{*}{400}  &$-3.8879$  (0.2614) &$-4.4601$ (0.2870)\\ 
				
				&\multirow{1}{*}{500}	 &  $-3.6765$ (0.4490)& $-4.6988$ (0.4773)&\\
				\hline
		\end{tabular}}
		\label{log_tri_data}
	\end{center}
\end{table}

The trivariate example relates to students' marks in $[0,20]$,  already foregrounded above with mean vector and covariance matrix amounting to $\widehat{\mu}=(7.609, 5.495, 7.147)^{\top}$ and  $\widehat{\Sigma}=\left(\begin{array}{rrr}49.271 & 30.341&33.148 \\30.341 & 28.295 & 26.600 \\33.148&26.600 &36.803\end{array}\right)$, respectively; see also Tables \ref{stat_desc_notes} and \ref{tri_data1}. The Gaussian kernel, utilizing a full bandwidth matrix via the plug-in method, proves to be less effective compared to the extended-beta smoother employing a diagonal bandwidth matrix obtained by means of Bayesian methods. However, it remains comparable in goodness-of-fit; see Table \ref{log_tri_data}. Notice that using a full bandwidth matrix achieves a certain level of smoothing that might not be achievable with diagonal matrices; see, e.g., \cite{Duong2007} and \cite{KS18} for comparisons with multivariate Gaussian and multiple beta kernels, respectively.

\begin{table}[!htbp]
	\centering
	\caption{Number of students per mark  out of 20 for 'Introduction to Practical Work' ($X_1$), 'Written and oral expression techniques' ($X_2$) and 'Scientific English' ($X_3$) for the first year ($n=590$) of the Mathematics-Physics-Chemistry-Computer Science program at Universit\'e Thomas Sankara.} 
	\begin{tabular}{cc|cc|cc}
		\hline Mark$_1$ ($X_1$) & $n_1$ & Mark$_2$ ($X_2$) & $n_2$ & Mark$_3$ ($X_3$) &  $n_3$ \\ \hline
		0 & 262 & 0 & 239 & 0 & 220 \\
		4 & 1 & 1 & 4 & 3 & 1 \\
		5 & 2 & 2 & 7 & 4 & 2 \\
		6 & 2 & 3 & 15 & 5 & 7 \\
		7 & 2 & 3,1 & 1 & 5,3 & 1 \\
		8 & 4 & 4 & 14 & 5,4 & 1 \\
		10 & 1 & 5 & 23 & 5,5 & 3 \\
		10,4 & 1 & 6 & 20 & 5,6 & 1 \\
		10,5 & 1 & 7 & 25 & 6 & 6 \\
		11 & 18 & 8 & 26 & 6,3 & 1 \\
		11,2 & 1 & 9 & 34 & 6,5 & 4 \\
		11,5 & 1 & 9,1 & 1 & 6,6 & 1 \\
		12 & 48 & 10 & 50 & 6,7 & 1 \\
		13 & 75 & 11 & 31 & 7 & 15 \\
		13,5 & 2 & 12 & 35 & 7,5 & 2 \\
		14 & 58 & 13 & 18 & 7,8 & 1 \\
		14,5 & 1 & 14 & 25 & 8 & 22 \\
		14,6 & 1 & 15 & 12 & 8,5 & 5 \\
		15 & 53 & 16 & 7 & 8,8 & 1 \\
		15,5 & 2 & 16,1 & 1 & 9 & 23 \\
		15,6 & 1 & 17 & 2 & 10 & 42 \\
		16 & 24 & & & 10,4 & 1 \\
		17 & 12 & & & 10,5 & 5 \\
		18 & 8 & & & 10,7 & 1 \\
		18,2 & 1 & & & 11 & 48 \\
		19 & 4 & & & 11,5 & 2 \\
		19,2 & 1 & & & 12 & 35 \\
		\multirow[t]{11}{*}{20} & 3 & & & 12,5 & 2 \\
		& & & & 13 & 33 \\
		& & & & 14 & 39 \\
		& & & & 15 & 26 \\
		& & & & 15,7 & 1 \\
		& & & & 16 & 19 \\
		& & & & 16,5 & 1 \\
		& & & & 17 & 8 \\
		& & & & 18 & 6 \\
		& & & & 19,5 & 1 \\
		& & & & 20 & 2 \\
		\hline
	\end{tabular}
	\label{tri_data1}
\end{table} 

\newpage
\section{Concluding remarks}\label{sec_conclusion}

In the current research paper, we have elaborated a nonparametric smoothing method using extended-beta kernels and both UCV and Bayesian adaptive selectors for the bandwidth vector. These (non-)normalized MEBK estimators have equally exhibited interesting asymptotic properties. We have explored the performance of this extended-beta smoother with Bayesian adaptive bandwidths for any density smoothing on either a given or estimated compact support. This can be regarded as a unified estimation of densities on bounded and unbounded domains. Efficient posterior and Bayes estimators for the bandwidth vector under a quadratic loss function have been explicitly derived.

Simulation studies and analysis of three real datasets have  yielded the powerful performance of the introduced approaches for nonparametric bandwidth estimation in terms of ISE and log-likelihood criteria. As expected, MEBK estimators with UCV or Bayesian adaptive bandwidths serve as workable  alternatives to Gaussian and gamma methods, granting  flexibility in selecting either global or variable bandwidth vectors. For an efficient as well as practical use of the MEBK estimator, it is recommended to assess the range of the multivariate data to make the optimal choice of compact support. 

At this stage, it is noteworthy that the proposed approach is promising and can be extended in several ways. This involves working on an \textsf{R} package dedicated to smoothing with MEBK, incorporating various tuning parameters, including Bayesian adaptive methods; see also the \textsf{Ake} package by \cite{Wansouwe2016}. Finally, and in order to reduce biases via modified versions of (standard) univariate extended-beta kernel estimators (e.g., \cite{Chen99}), another way for improving this work would be to investigate some effects of the so-called combined MEBK estimations in the same sense of multiple combined gamma kernel estimations in \cite{SomeEtAl2024}.


\section{Appendix: Proofs}
\label{sec_appendix proofs}

\begin{proof}[Proof of Lemma \ref{lemma}]
Departing from \eqref{noyaugamma} and relying upon the following version of the beta function 
$$B(r,s)=\int_a^b\frac{(u-a)^{r-1}(b-u)^{s-1}}{(b-a)^{r+s-1}}du,$$ 
for all $a<b$, $r>0$ and $s>0$, we have
\begin{eqnarray*}
\left|\left|\prod_{j=1}^{d}{EB}_{x_j,h_j,a_j,b_j}\right|\right|_2^2 
&:=&\!\!\!\prod_{j=1}^{d}\int_{[a_j,b_j]}{EB}_{x_j,h_j,a_j,b_j}^2(u)du \\
&=&\!\!\!\prod_{j=1}^{d}\!\int_{a_j}^{b_j}\!
\frac{(u-a_j)^{2(x_j-a_j)/\{(b_j-a_j)h_j\}}(b_j-u)^{2(b_j-x_j)/\{(b_j-a_j)h_j\}}(b_j-a_j)^{-2-2/h_j}}{B^2\{1+(x_j-a_j)/(b_j-a_j)h_j,1+(b_j-x_j)/(b_j-a_j)h_j\}}du\\
&=&\!\!\!\prod_{j=1}^{d}\, \frac{B\left(1+2(x_j-a_j)/(b_j-a_j)h_j,1+2(b_j-x_j)/(b_j-a_j)h_j\right)/(b_j-a_j)}
{\left[B\left(1+(x_j-a_j)/(b_j-a_j)h_j,1+(b_j-x_j)/(b_j-a_j)h_j\right)\right]^{2}}; 
\end{eqnarray*}
which is the desired result of the first part of the lemma. 
The last parts of the lemma are trivial, which are also inferred from \eqref{eq:AxH} and \eqref{eq:BxH}, with $||\mathbf{A}||^2:=\mathbf{A}^\top\mathbf{A}$.
\end{proof}

\begin{proof}[Proof of Proposition \ref{prop_BVf}]	
One can apply Proposition 2.9 of \cite{KS18} with both characteristics \eqref{eq:AxH} and \eqref{eq:BxH} of the MEBK estimator. Thus, the pointwise bias is easily verified. Furthermore, the pointwise variance is deduced from Lemma \ref{lemma} and Eq. (20) of \cite{LK17} with $r_2=1/2$.
\end{proof}

\begin{proof}[Proof of Proposition \ref{propo_L2asNorm}]
Since the bandwidth matrix $\mathbf{H}$	is here diagonal and resting on both characteristics \eqref{eq:AxH} and \eqref{eq:BxH} in addition to Lemma \ref{lemma} of the MEBK estimator, one can deduce both results of mean square and almost surely convergences from Theorem 2.2 of \cite{KL18} with $r_2=1/2$.
 
As for the convergence in law, one can consider Theorem 3.3 of \cite{Esstafa23b} in its multiple versions as follows. Writing 
\begin{eqnarray*}
\sqrt{n}\left(\prod_{j=1}^d h_j^{(1/2)\alpha_j}(n)\right)\left(\widehat{f}_{n}(\boldsymbol{x})-f(\boldsymbol{x})\right) &=& \frac{\widehat{f}_{n}(\boldsymbol{x})-\mathbb{E}\widehat{f}_{n}(\boldsymbol{x})}{\sqrt{Var\widehat{f}_{n}(\boldsymbol{x})}}\sqrt{n\left(\prod_{j=1}^d h_j^{\alpha_j}(n)\right) Var\widehat{f}_{n}(\boldsymbol{x})}\\
& &+\sqrt{n\left(\prod_{j=1}^d h_j^{\alpha_j}(n)\right)}\left[\mathbb{E}\widehat{f}_{n}(\boldsymbol{x})-f(\boldsymbol{x})\right]
\end{eqnarray*}
with $\sqrt{n\left(\prod_{j=1}^d h_j^{\alpha_j}(n)\right)}\left[\mathbb{E}\widehat{f}_{n}(\boldsymbol{x})-f(\boldsymbol{x})\right]=\mathcal{O}\left(\sqrt{n}\prod_{j=1}^d h_j^{(3/2)\alpha_j}(n)\right)$
and 
$$n\left(\prod_{j=1}^d h_j^{\alpha_j}(n)\right) Var\widehat{f}_{n}(\boldsymbol{x})=f(\boldsymbol{x})\left(\prod_{j=1}^d h_j^{\alpha_j}(n)\right)\left|\left|\prod_{j=1}^{d}{EB}_{x_j,h_j(n),a_j,b_j}\right|\right|_2^2
+\mathcal{O}\left(\prod_{j=1}^d h_j^{\alpha_j}(n)\right),
$$
one deduces 
$[\widehat{f}_{n}(\boldsymbol{x})-\mathbb{E}\widehat{f}_{n}(\boldsymbol{x})]/\sqrt{Var\widehat{f}_{n}(\boldsymbol{x})}\xrightarrow[n\to \infty]{\mathcal{L}} \mathcal{N}(0,1)$ 
resting on triangular array technique with the Lyapunov condition.
\end{proof}

\begin{proof}[Proof of Proposition \ref{propo_Cn}]
Following the proof of Proposition 3.4 of \cite{Esstafa23b}, we first rewrite 
$$\mathbb{E}(C_n-1)^2=Var C_n + (\mathbb{E}C_n -1)^2.$$
Therefore, one demonstrates that $Var C_n\underset{n\to \infty}{\longrightarrow} 0$ since one can successively get 
\begin{eqnarray*}
Var C_n &\leq & \frac{1}{n}\left|\left|\prod_{j=1}^{d}{EB}_{x_j,h_j,a_j,b_j}\right|\right|_2^2\int_{\boldsymbol{y}\in\mathbb{T}_d}f(\boldsymbol{y})d\boldsymbol{y}\\
& &+\frac{c}{n}\int_{\boldsymbol{y}\in\mathbb{T}_d}f(\boldsymbol{y})d\boldsymbol{y}
+\frac{c}{n}\int_{\boldsymbol{z}\in\mathbb{T}_d}\int_{\boldsymbol{y}\in\mathbb{T}_d}f(\boldsymbol{y})f(\boldsymbol{z})d\boldsymbol{y}d\boldsymbol{z}\\
&\leq & \frac{c}{n\prod_{j=1}^d h_j^{\alpha_j}(n)} + \frac{c}{n},
\end{eqnarray*}
where $c$ denotes our generic constant that can change from line to line.
Finally, one gets $(\mathbb{E}C_n -1)^2\underset{n\to \infty}{\longrightarrow} 0$ since $\mathbb{E}\widehat{f}_{n}(\boldsymbol{x})-f(\boldsymbol{x})=\mathcal{O}\left(\prod_{j=1}^d h_j^{\alpha_j}(n)\right)$ and 
$$|\mathbb{E}(C_n-1)|\leq\int_{\boldsymbol{x}\in\mathbb{T}_d}|\mathbb{E}\left(\widehat{f}_{n}(\boldsymbol{x})-f(\boldsymbol{x})\right)|d\boldsymbol{x}
\underset{n\to \infty}{\longrightarrow} 0.
$$
This completes the proof.
\end{proof}

\begin{proof}[Proof of Proposition \ref{propo_L2Norm}]	
Since $\widehat{f}_{n}(\boldsymbol{x})$ converges in mean square to $f(\boldsymbol{x})$ through the first part of Proposition \ref{propo_L2asNorm}, the result of the convergence in probability is easily inferred from 
$$\widetilde{f}_{n}(\boldsymbol{x})-f(\boldsymbol{x})
=\frac{1}{C_n}\left\{\left(\widehat{f}_{n}(\boldsymbol{x})-f(\boldsymbol{x})\right)+(1-C_n)f(\boldsymbol{x})\right\}
$$
and Proposition \ref{propo_Cn} with the Slutsky theorem.

As for the second part of the asymptotic normality, we use  the multiple version of Theorem 3.6 of \cite{Esstafa23b} through decomposing
\begin{eqnarray*}
\sqrt{n}\left(\prod_{j=1}^d h_j^{(1/2)\alpha_j}(n)\right)\left(\widetilde{f}_{n}(\boldsymbol{x})-f(\boldsymbol{x})\right) 
&=& \frac{1}{C_n}\frac{\widehat{f}_{n}(\boldsymbol{x})-\mathbb{E}\widehat{f}_{n}(\boldsymbol{x})}{\sqrt{Var\widehat{f}_{n}(\boldsymbol{x})}}\sqrt{n\left(\prod_{j=1}^d h_j^{\alpha_j}(n)\right) Var\widehat{f}_{n}(\boldsymbol{x})}\\
& &+\frac{1}{C_n}\sqrt{n\left(\prod_{j=1}^d h_j^{\alpha_j}(n)\right)}\left[\mathbb{E}\widehat{f}_{n}(\boldsymbol{x})-f(\boldsymbol{x})\right]\\
& &+\frac{f(\boldsymbol{x})}{C_n}\sqrt{n\left(\prod_{j=1}^d h_j^{\alpha_j}(n)\right)}\:\left(1-C_n\right)\\
&=& \xrightarrow[n\to \infty]{\mathcal{L}} \mathcal{N}(0,f(\boldsymbol{x})\lambda_{\boldsymbol{x},\boldsymbol{\alpha}}) \:
+ \xrightarrow[n\to \infty]{\mathbb{P}} 0 \: 
+ \xrightarrow[n\to \infty]{\mathbb{P}} 0.
\end{eqnarray*}
The last convergence in probability is obtained by
$$\sqrt{n\left(\prod_{j=1}^d h_j^{\alpha_j}(n)\right)}\:\mathbb{E}\left(1-C_n\right)=\mathcal{O}\left(\sqrt{n}\prod_{j=1}^d h_j^{(3/2)\alpha_j}(n)\right)
$$
and the Hoeffding inequality.
\end{proof}

\begin{proof}[Proof of Proposition \ref{propo_Comparison}]	
Following the proof of Proposition 3.8 in \cite{Esstafa23b}, we can characterize this multiple case by easily admitting that $C_n$ tends towards 1 for $L^4$ criterion of the MEBK estimator via Lemma \ref{lemma}. 

Thus, we successively get 
\begin{eqnarray*}
\mathbb{E}\left[\int_{\mathbb{T}_d}\left|\widetilde{f}_{n}(\boldsymbol{x})-f(\boldsymbol{x})\right|^2d\boldsymbol{x}\right] 
&=& \mathbb{E}\left[\int_{\mathbb{T}_d}\left|\frac{\widehat{f}_{n}(\boldsymbol{x})}{C_n}-\frac{f(\boldsymbol{x})}{C_n}+\frac{f(\boldsymbol{x})}{C_n}-f(\boldsymbol{x})\right|^2d\boldsymbol{x}\right] \\
&=& \mathbb{E}\left[\int_{\mathbb{T}_d}\left|\widehat{f}_{n}(\boldsymbol{x})-f(\boldsymbol{x})\right|^2d\boldsymbol{x}\right] 
+ \mathbb{E}\left[\left|\frac{1-C_n}{C_n}\right|^2\right]\int_{\mathbb{T}_d}f(\boldsymbol{x})d\boldsymbol{x}\\
& & + \mathbb{E}\left[\frac{1-C_n^2}{C_n^2}\int_{\mathbb{T}_d}\left|\widehat{f}_{n}(\boldsymbol{x})-f(\boldsymbol{x})\right|^2d\boldsymbol{x}\right]\\
& & + 2\, \mathbb{E}\left[C_n^{-2}\int_{\mathbb{T}_d}\{\widehat{f}_{n}(\boldsymbol{x})-f(\boldsymbol{x})\}(1-C_n)f(\boldsymbol{x})d\boldsymbol{x}\right].
\end{eqnarray*}
This gives rise to these inqualities after a few consecutive increases
\begin{eqnarray*}
\mathbb{E}\left[\int_{\mathbb{T}_d}\left|\widetilde{f}_{n}(\boldsymbol{x})-f(\boldsymbol{x})\right|^2d\boldsymbol{x}\right] 
&\leq & \mathbb{E}\left[\int_{\mathbb{T}_d}\left|\widehat{f}_{n}(\boldsymbol{x})-f(\boldsymbol{x})\right|^2d\boldsymbol{x}\right] 
+ c\,\mathbb{E}\left[\left|\frac{1-C_n}{C_n}\right|^2\right]\\
& & + c\,\mathbb{E}\left[\frac{\left|1-C_n^2\right|}{C_n^2}\right]
+ c\,\mathbb{E}\left[\frac{\left|1-C_n\right|(C_n+1)}{C_n^2}\right]\\
&\leq & \mathbb{E}\left[\int_{\mathbb{T}_d}\left|\widehat{f}_{n}(\boldsymbol{x})-f(\boldsymbol{x})\right|^2d\boldsymbol{x}\right] 
+ c\sqrt{\mathbb{E}[(C_n-1)^4]\mathbb{E}[C_n^{-4}]}\\
& & + c\sqrt{\mathbb{E}[(C_n-1)^4]\left\{\mathbb{E}[C_n^{-2}]+\mathbb{E}[C_n^{-3}]+\mathbb{E}[C_n^{-4}]\right\}},
\end{eqnarray*}
where $c$ corresponds to our generic constant that can change from line to line. Therefore, one has to deduce the desired result.
\end{proof}

\begin{proof}[Proof of Theorem \ref{theo_Bayes}]
	(i) Relying upon \eqref{leaveoneout} and \eqref{prior}, the numerator of \eqref{bayesrule} is first equal to
	\begin{eqnarray}
		N(\boldsymbol{h}_{i}\mid\mathbf{X}_{i})&=&\widehat{f}_{-i}(\mathbf{X}_{i}\mid\boldsymbol{h}_{i}) \pi (\boldsymbol{h}_{i})\nonumber\\
		&=& \left(\frac{1}{n-1}\sum_{j=1,j\neq i}^{n} \prod_{\ell=1}^{d} BE_{X_{i\ell},h_{i\ell},a_{\ell},b_{\ell}}(X_{j\ell})\right)\left(\prod_{\ell=1}^{d} \frac{\beta_{\ell}^{\alpha}}{\Gamma(\alpha)}h_{i\ell}^{-\alpha-1}\exp(-\beta_{\ell}/h_{i\ell})\right)\nonumber\\
		&=&\frac{[\Gamma(\alpha)]^{-d}}{n-1}\sum_{j=1,j\neq i}^{n} \prod_{\ell=1}^{d} \frac{BE_{X_{i\ell},h_{i\ell},a_{\ell},b_{\ell}}(X_{j\ell})}{\beta_{\ell}^{-\alpha}h_{i\ell}^{\alpha+1}\exp(\beta_{\ell}/h_{i\ell})}\label{equ8}
		.
	\end{eqnarray}
Grounded on \eqref{noyaugamma} and applying the usual properties $B(s,t)=\Gamma(s)\Gamma(t)/\Gamma(s+t)$ and $\Gamma(z+1)=z\Gamma(z)$ for $z>0$, part of expression \eqref{equ8} becomes
	\begin{eqnarray}\label{N1}
		\frac{BE_{X_{i\ell},h_{i\ell},a_{\ell},b_{\ell}}(X_{j\ell})}{\beta_{\ell}^{-\alpha}h_{i\ell}^{\alpha+1}\exp(\beta_{\ell}/h_{i\ell})}
		&=&\frac{(X_{j\ell}-a_{\ell})^{(X_{i\ell}-a_{\ell})/[(b_{\ell}-a_{\ell})h_{i\ell}]}(b_{\ell}-X_{j\ell})^{(b_{\ell}-X_{i\ell})/[(b_{\ell}-a_{\ell})h_{i\ell}]}}{B(1+(X_{i\ell}-a_{\ell})/(b_{\ell}-a_{\ell})h_{i\ell},1+(b_{\ell}-X_{i\ell})/(b_{\ell}-a_{\ell})h_{i\ell})}\nonumber\\
		&&\times~~(b_{\ell}-a_{\ell})^{(-1-1/h_{i\ell})}\beta_{\ell}^{\alpha}h_{i\ell}^{-\alpha-1}\exp(-\beta_{\ell}/h_{i\ell})\nonumber\\
		&=&\frac{(1+1/h_{i\ell})\Gamma(1+1/h_{i\ell})}{\Gamma(1+(X_{i\ell}-a_{\ell})/(b_{\ell}-a_{\ell})h_{i\ell})\Gamma(1+(b_{\ell}-X_{i\ell})/(b_{\ell}-a_{\ell})h_{i\ell}) }\nonumber\\
		&&\times~~\frac{(X_{j\ell}-a_{\ell})^{(X_{i\ell}-a_{\ell})/[(b_{\ell}-a_{\ell})h_{i\ell}]}(b_{\ell}-X_{j\ell})^{(b_{\ell}-X_{i\ell})/[(b_{\ell}-a_{\ell})h_{i\ell}]}}{(b_{\ell}-a_{\ell})^{(1+1/h_{i\ell})}\beta_{\ell}^{-\alpha}h_{i\ell}^{\alpha+1}\exp(\beta_{\ell}/h_{i\ell})}.
	\end{eqnarray}
	
Consider the largest part $\mathbb{I}^{c}_{i}=\left \{\ell \in \{1,\ldots,d\}~;X_{i\ell} \in (a_\ell, b_\ell)\right \}$. Following \cite{Chen99,Chenn2000a}, we assume that for all $X_{i\ell} \in (a_\ell, b_\ell)$, one has $(X_{i\ell}-a_\ell)/h_\ell \rightarrow \infty$ and $(b_\ell-X_{i\ell})/h_\ell \rightarrow \infty$  as $n \rightarrow \infty$  for all $\ell \in  {1, 2, \ldots, d}$. As a matter of fact, it follows from the  Sterling formula $\Gamma(z+1)\simeq\sqrt{2\pi}z^{z+1/2}\exp(-z)$ as $z\rightarrow \infty$ that as $n\rightarrow \infty$, the previous term \eqref{N1} can be successively calculated as
	\begin{eqnarray}\label{N1a}
		\frac{BE_{X_{i\ell},h_{i\ell},a_{\ell},b_{\ell}}(X_{j\ell})}{\beta_{\ell}^{-\alpha}h_{i\ell}^{\alpha+1}\exp(\beta_{\ell}/h_{i\ell})}&=&\frac{\{(X_{i\ell}-a_\ell)(b_\ell-X_{j\ell})\}^{-1/2}}{\sqrt{2\pi}\beta_{\ell}^{-\alpha}\exp(\beta_{\ell}/h_{i\ell})}\left(h^{-\alpha-3/2}_{i\ell}+h^{-\alpha-1/2}_{i\ell}\right)\exp\left[\left(\frac{1}{h_{i\ell}(b_\ell-a_\ell)}\right)\right.\nonumber\\
		&&\times\left.\left((X_{i\ell}-a_{\ell})\log\left(\frac{X_{j\ell}-a_{\ell}}{X_{i\ell}-a_{\ell}}\right)+(b_\ell-X_{i\ell})\log\left(\frac{b_\ell-X_{j\ell}}{b_\ell-X_{i\ell}}\right)\right)\right]\nonumber\\
		&=&\frac{\{(X_{i\ell}-a_\ell)(b_\ell-X_{j\ell})\}^{-1/2}}{\sqrt{2\pi}\beta_{\ell}^{-\alpha}} \nonumber \\
		&&\times\left\{\frac{\Gamma(\alpha+1/2)}{[B_{ij\ell}(\alpha,\beta_\ell)]^{\alpha+1/2}} 
	\times \frac{[B_{ij\ell}(\alpha,\beta_\ell)]^{\alpha+1/2}\exp[-B_{ij\ell}(\alpha,\beta_\ell)/h_{i\ell}]}{h_{i\ell}^{\alpha+3/2}\Gamma(\alpha+1/2)}\right.\nonumber \\
		&&\quad\left. +\frac{\Gamma(\alpha-1/2)}{[B_{ij\ell}(\alpha,\beta_\ell)]^{\alpha-1/2}}
	\times \frac{[B_{ij\ell}(\alpha,\beta_\ell)]^{\alpha-1/2}\exp[-B_{ij\ell}(\alpha,\beta_\ell)/h_{i\ell}]}{h_{i\ell}^{\alpha+1/2}\Gamma(\alpha-1/2)}\right\} \nonumber\\
		&=&\!\!\!A_{ij\ell}(\alpha,\beta_\ell) IG_{\alpha+1/2,B_{ij\ell}(\alpha,\beta_\ell)}(h_{i\ell}) +C_{ij\ell}(\alpha,\beta_\ell) IG_{\alpha-1/2,B_{ij\ell}(\alpha,\beta_\ell)}(h_{i\ell}),
	\end{eqnarray}
where $A_{ij\ell}(\alpha,\beta_\ell)$, $B_{ij\ell}(\alpha,\beta_\ell)$ and $C_{ij\ell}(\alpha,\beta_\ell)$ are as stated in the theorem, and both $IG_{\alpha+1/2,B_{ij\ell}(\alpha,\beta_\ell)}(h_{i\ell})$ and $IG_{\alpha-1/2,B_{ij\ell}(\alpha,\beta_\ell)}(h_{i\ell})$ are easily deduced from \eqref{prior}. 
	
	Moreover, considering the left smallest part $\mathbb{I}_{ia_k}=\left \{k \in \{1,\ldots,d\}~;X_{ik} = a_k\right \}$, for each $X_{ik} = a_k$, the term of sum \eqref{equ8}   in \eqref{N1} can be expressed as follows:	
	\begin{eqnarray}\label{N1b}
		\frac{BE_{a_{k},h_{ik},a_{k},b_{k}}(X_{jk})}{\beta_{k}^{-\alpha}h_{ik}^{\alpha+1}\exp(\beta_{k}/h_{ik})}
		&=&\frac{(1+1/h_{ik})(b_{k}-X_{jk})^{(1/h_{ik})}}{ (b_{k}-a_{k})^{(1+1/h_{ik})}\beta_{k}^{-\alpha}h_{ik}^{\alpha+1}\exp(\beta_{k}/h_{ik})}\nonumber\\
		&=&\frac{\beta_{k}^{\alpha}}{(b_k-a_k)}\left(h_{ik}^{-\alpha-1}+h_{ik}^{-\alpha-2}\right)\exp\left\{-\left[\beta_{k}- \log\left(\frac{b_k-X_{jk}}{b_k-a_{k}}\right)\right]/h_{ik}\right\}\nonumber\\
		&=&\left(\frac{\Gamma(\alpha)}{[E_{jk}(\beta_k)]^{\alpha}} \frac{[E_{jk}(\beta_k)]^{\alpha}}{\Gamma(\alpha)}h^{-\alpha-1}  +\frac{\Gamma(\alpha+1)}{[E_{jk}(\beta_k)]^{\alpha+1}} \frac{[E_{jk}(\beta_k)]^{\alpha+1}}{\Gamma(\alpha+1)}h^{-\alpha-2}\right)\nonumber\\
		&&\times \frac{\beta_{k}^{\alpha}}{(b_k-a_k)}\exp[-E_{jk}(\beta_k)/h_{ik}]\nonumber\\
		&=& F_{jk}(\alpha,\beta_k)\,IG_{\alpha,E_{jk}(\beta_k)}(h_{ik})+H_{jk}(\alpha,\beta_k)\,IG_{\alpha+1,E_{jk}(\beta_k)}(h_{ik})
	\end{eqnarray}
where $E_{jk}(\beta_k)$, $F_{jk}(\alpha,\beta_k)$ and $H_{jk}(\alpha,\beta_k)$ are as presented in the theorem and, both $IG_{\alpha,E_{jk}(\beta_k)}(h_{ik})$ and $IG_{\alpha+1,E_{jk}(\beta_k)}(h_{ik})$ derive from \eqref{prior}.
	
	Similarly, by focusing on the right smallest part $\mathbb{I}_{ib_s}=\left \{s \in \{1,\ldots,d\}~;X_{is} = b_s\right \}$, for each $X_{ik} = b_k$, the term of sum \eqref{equ8} in \eqref{N1} can be indicated as follows:	
	\begin{eqnarray}\label{N1c}
		\frac{BE_{b_{s},h_{is},a_{s},b_{s}}(X_{js})}{\beta_{s}^{-\alpha}h_{is}^{\alpha+1}\exp(\beta_{s}/h_{is})}
		&=&\frac{(1+1/h_{is})(X_{js}-a_s)^{(1/h_{is})}}{ (b_{s}-a_{s})^{(1+1/h_{is})}\beta_{s}^{-\alpha}h_{is}^{\alpha+1}\exp(\beta_{s}/h_{is})}\nonumber\\
		&=&\frac{\beta_{s}^{\alpha}}{(b_s-a_s)}\left(h_{is}^{-\alpha-1}+h_{is}^{-\alpha-2}\right)\exp\left\{-\left[\beta_{s}- \log\left(\frac{X_{js}-a_s}{b_s-a_{s}}\right)\right] h^{-1}_{is}\right\}\nonumber\\
		&=&\left(\frac{\Gamma(\alpha)}{[G_{js}(\beta_s)]^{\alpha}} \frac{[G_{js}(\beta_s)]^{\alpha}}{\Gamma(\alpha)}h^{-\alpha-1}  +\frac{\Gamma(\alpha+1)}{[G_{js}(\beta_s)]^{\alpha+1}} \frac{[G_{js}(\beta_s)]^{\alpha+1}}{\Gamma(\alpha+1)}h^{-\alpha-2}\right)\nonumber\\
		&&\times\frac{\beta_{s}^{\alpha}}{(b_s-a_s)} \exp[-G_{js}(\beta_s)/h_{is}]\nonumber\\
		&=& J_{js}(\alpha,\beta_s)\,IG_{\alpha,G_{js}(\beta_s)}(h_{is})+K_{js}(\alpha,\beta_s)\,IG_{\alpha+1,G_{js}(\beta_k)}(h_{is})
	\end{eqnarray}
where $G_{js}(\beta_s)$, $J_{js}(\alpha,\beta_s)$ and 	$K_{js}(\alpha,\beta_s)$ are as defined in the thoerem and, both $IG_{\alpha,G_{js}(\beta_s)}(h_{is})$ and $IG_{\alpha+1,G_{js}(\beta_k)}(h_{is})$ are deduced from $(\ref{prior})$.
	
	Combining \eqref{N1a}, \eqref{N1b} and \eqref{N1c}, the expression of $N(\boldsymbol{h}_{i}\mid\mathbf{X}_{i})$ in $(\ref{equ8})$ becomes
	\begin{eqnarray}\label{equ11}
		N(\boldsymbol{h}_{i}\mid\mathbf{X}_{i})&=&\frac{[\Gamma(\alpha)]^{-d}}{n-1}\sum_{j=1,j\neq i}^{n} \left(\prod_{k \in \mathbb{I}_{ia_k}}[F_{jk}(\alpha,\beta_k)\,IG_{\alpha,E_{jk}(\beta_k)}(h_{ik})+H_{jk}(\alpha,\beta_k)\,IG_{\alpha+1,E_{jk}(\beta_k)}(h_{ik})]\right) \nonumber\\
		&&\quad \times \left(\prod_{\ell \in \mathbb{I}_{i}^c}[ A_{ij\ell}(\alpha,\beta_\ell)\,IG_{\alpha+1/2,B_{ij\ell}(\alpha,\beta_\ell)}(h_{i\ell})+C_{ij\ell}(\alpha,\beta_\ell)\,IG_{\alpha-1/2,B_{ij\ell}(\alpha,\beta_\ell)}(h_{i\ell})]\right)\nonumber\\
		&&\quad \times  \left(\prod_{k \in \mathbb{I}_{ib_s}} [J_{js}(\alpha,\beta_s)\,IG_{\alpha,G_{js}(\beta_s)}(h_{is})+K_{js}(\alpha,\beta_s)\,IG_{\alpha+1,G_{js}(\beta_k)}(h_{is})]\right).
	\end{eqnarray}
Referring to $(\ref{equ11})$, the denominator of \eqref{bayesrule} is successively determined as follows
	\begin{eqnarray}\label{equ12}
		\int_{\times_{\ell=1}^{d}[a_{\ell}, b_{\ell}]} N(\boldsymbol{h}_{i}\mid\mathbf{X}_{i})\,d\boldsymbol{h}_{i}\nonumber
		&=&\frac{[\Gamma(\alpha)]^{-d}}{(n-1)}\sum_{j=1,j\neq i}^{n} \left(\prod_{k \in \mathbb{I}_{ia_k}}\left[ F_{jk}(\alpha,\beta_k)\,\int_{a_{k}}^{b_{k}}IG_{\alpha,E_{jk}(\beta_k)}(h_{ik})dh_{ik}\right.\right.\\
		&&\left.\left.+H_{jk}(\alpha,\beta_k)\,\int_{a_{k}}^{b_{k}}IG_{\alpha+1,E_{jk}(\beta_k)}(h_{ik})dh_{ik}\right]\right) \nonumber\\
		&&\quad \times \left(\prod_{\ell \in \mathbb{I}_{i}^c} \left[ A_{ij\ell}(\alpha,\beta_\ell)\,\int_{a_{\ell}}^{b_{\ell}}IG_{\alpha+1/2,B_{ij\ell}(\alpha,\beta_\ell)}(h_{i\ell})dh_{i\ell}\right.\right.\nonumber\\
		&&\left.\left.+C_{ij\ell}(\alpha,\beta_\ell)\,\int_{a_{\ell}}^{b_{\ell}}IG_{\alpha-1/2,B_{ij\ell}(\alpha,\beta_\ell)}(h_{i\ell})dh_{i\ell}\right]\right)\nonumber\\
		&&\quad \times  \left(\prod_{k \in \mathbb{I}_{ib_s}} \left[ J_{js}(\alpha,\beta_s)\,\int_{a_{s}}^{b_{s}}IG_{\alpha,G_{js}(\beta_s)}(h_{is})dh_{is}\right.\right.\nonumber\\
		&&\left.\left.+H_{js}(\alpha,\beta_s)\,\int_{a_{s}}^{b_{s}}IG_{\alpha+1,G_{js}(\beta_s)}(h_{is})dh_{is}\right]\right)\nonumber\\
		&=&\frac{[\Gamma(\alpha)]^{-d}}{(n-1)}\sum_{j=1,j\neq i}^{n} \left(\prod_{k \in \mathbb{I}_{ia_k}}[F_{jk}(\alpha,\beta_k)+K_{jk}(\alpha,\beta_k)]\right) \nonumber\\
		&&\quad \times \left(\prod_{\ell \in \mathbb{I}_{i}^c} [A_{ij\ell}(\alpha,\beta_\ell)+C_{ij\ell}(\alpha,\beta_\ell)]\right)\nonumber\\
		&&\quad \times  \left(\prod_{k \in \mathbb{I}_{ib_s}} [J_{js}(\alpha,\beta_s)+K_{js}(\alpha,\beta_s)]\right)\nonumber\\
		&=&\frac{[\Gamma(\alpha)]^{-d}}{(n-1)}D_{i}(\alpha,\boldsymbol{\beta}),
	\end{eqnarray}
with $D_{i}(\alpha,\boldsymbol{\beta})$ being as given in the theorem. Thus, the ratios of \eqref{equ11} and \eqref{equ12} allow us to conclude Part (i) of the theorem.
	
	(ii) We recall that the mean of the inverse-gamma distribution  $\mathcal{IG}(\alpha,\beta_\ell)$ is $\beta_\ell/(\alpha-1)$ and $\mathbb{E}(h_{i\ell}\mid\mathbf{X}_{i})=\int_{a_m}^{b_m} h_{i\ell}\pi(h_{im}\mid\mathbf{X}_{i})\,dh_{im}$ with $\pi(h_{im}\mid\mathbf{X}_{i})$ referring to the marginal distribution $h_{im}$ obtained by integration of $\pi(\boldsymbol{h}_{i}\mid\mathbf{X}_{i})$  for all components of  $\boldsymbol{h}_{i}$ except $h_{im}$. Then, $\pi(h_{im}\mid\mathbf{X}_{i})=\int_{\mathbb{T}_{d(-m)}}\pi(\boldsymbol{h}_{i}\mid\mathbf{X}_{i})\,d\boldsymbol{h}_{i(-m)}$, where $\mathbb{T}_{d(-m)}=\times_{\ell=1,\ell\neq m}^{d}[a_{\ell}, b_{\ell}]$ and  $d\boldsymbol{h}_{i(-m)}$ is the vector  $d\boldsymbol{h}_{i}$ without the $m$-th component.
	
	If $m\in \mathbb{I}_{i}^c$ and $\alpha>3/2$, one has 	
	\begin{eqnarray*}\label{equ13}
		\pi(h_{im}\mid\mathbf{X}_{i})
		&=& \frac{1}{D_{i}(\alpha,\boldsymbol{\beta})}\sum_{j=1,j\neq i}^{n} \left(\prod_{k \in \mathbb{I}_{ia_k}}\left[ F_{jk}(\alpha,\beta_k)+H_{jk}(\alpha,\beta_k)\right]\right) \nonumber\\
		&&\times \left(\prod_{\ell \in \mathbb{I}_{i}^c, \ell \neq m} \left[ A_{ij\ell}(\alpha,\beta_\ell)+C_{ij\ell}(\alpha,\beta_\ell)\right]\right)\left(\prod_{k \in \mathbb{I}_{ib_s}} \left[ J_{js}(\alpha,\beta_s)+K_{js}(\alpha,\beta_s)\right]\right)\nonumber\\
		&&\times  \left(A_{ijm}(\alpha,\beta_m)IG_{\alpha+1/2,B_{ijm}(\alpha,\beta_m)}(h_{im})+C_{ijm}(\alpha,\beta_m)IG_{\alpha-1/2,B_{ijm}(\alpha,\beta_m)}(h_{im})\right).\nonumber\\
	\end{eqnarray*}
As a matter of fact,	
	\begin{eqnarray}\label{h1}
		\widetilde{h}_{im}&=&\frac{1}{D_{i}(\alpha,\boldsymbol{\beta})}\sum_{j=1,j\neq i}^{n} \left(\prod_{k \in \mathbb{I}_{ia_k}}\left[ F_{jk}(\alpha,\beta_k)+H_{jk}(\alpha,\beta_k)\right]\right) \nonumber\\
		&&\times \left(\prod_{\ell \in \mathbb{I}_{i}^c} \left[ A_{ij\ell}(\alpha,\beta_\ell)+C_{ij\ell}(\alpha,\beta_\ell)\right]\right)\left(\prod_{k \in \mathbb{I}_{ib_s}} \left[ J_{js}(\alpha,\beta_s)+K_{js}(\alpha,\beta_s)\right]\right)\nonumber\\
		&&\times  \left(\frac{A_{ijm}(\alpha,\beta_m)}{\alpha-1/2}+\frac{C_{ijm}(\alpha,\beta_m)}{\alpha-3/2}\right)\frac{B_{ijm}(\alpha,\beta_m)}{A_{ijm}(\alpha,\beta_m)+C_{ijm}(\alpha,\beta_m)} .
	\end{eqnarray}
	
If $m\in \mathbb{I}_{ia_k}$ and $\alpha>1$, one gets	
	\begin{eqnarray*}\label{equ13}
		\pi(h_{im}\mid\mathbf{X}_{i})
		&=& \frac{1}{D_{i}(\alpha,\boldsymbol{\beta})}\sum_{j=1,j\neq i}^{n} \left(\prod_{k \in \mathbb{I}_{ia_k}, k\neq m}\left[ F_{jk}(\alpha,\beta_k)+H_{jk}(\alpha,\beta_k)\right]\right) \nonumber\\
		&&\times \left(\prod_{\ell \in \mathbb{I}_{i}^c} \left[ A_{ij\ell}(\alpha,\beta_\ell)+C_{ij\ell}(\alpha,\beta_\ell)\right]\right)\left(\prod_{k \in \mathbb{I}_{ib_s}} \left[ J_{js}(\alpha,\beta_s)+K_{js}(\alpha,\beta_s)\right]\right)\nonumber\\
		&&\times  \left(F_{jm}(\alpha,\beta_m)IG_{\alpha,E_{jm}(\beta_m)}(h_{im})+H_{jm}(\alpha,\beta_m)IG_{\alpha+1,E_{jm}(\beta_m)}(h_{im})\right)\nonumber\\
	\end{eqnarray*}
and, consequently	
	\begin{eqnarray}\label{h2}
		\widetilde{h}_{im}&=& \frac{1}{D_{i}(\alpha,\boldsymbol{\beta})}\sum_{j=1,j\neq i}^{n} \left(\prod_{k \in \mathbb{I}_{ia_k} }\left[ F_{jk}(\alpha,\beta_k)+H_{jk}(\alpha,\beta_k)\right]\right) \nonumber\\
		&&\times \left(\prod_{\ell \in \mathbb{I}_{i}^c} \left[ A_{ij\ell}(\alpha,\beta_\ell)+C_{ij\ell}(\alpha,\beta_\ell)\right]\right)\left(\prod_{k \in \mathbb{I}_{ib_s}} \left[ J_{js}(\alpha,\beta_s)+K_{js}(\alpha,\beta_s)\right]\right)\nonumber\\
		&&\times  \left(\frac{F_{jm}(\alpha,\beta_m)}{\alpha-1}+\frac{H_{jm}(\alpha,\beta_m)}{\alpha}\right) \frac{E_{jm}(\beta_m)}{F_{jm}(\alpha,\beta_m)+H_{jm}(\alpha,\beta_m)} .
	\end{eqnarray}	
	
	If $m\in  \mathbb{I}_{ib_s}$ and $\alpha>0$, one obtains	
	\begin{eqnarray*}\label{equ13}
		\pi(h_{im}\mid\mathbf{X}_{i})
		&=& \frac{1}{D_{i}(\alpha,\boldsymbol{\beta})}\sum_{j=1,j\neq i}^{n} \left(\prod_{k \in \mathbb{I}_{ia_k}}\left[ F_{jk}(\alpha,\beta_k)+H_{jk}(\alpha,\beta_k)\right]\right) \nonumber\\
		&&\times \left(\prod_{\ell \in \mathbb{I}_{i}^c} \left[ A_{ij\ell}(\alpha,\beta_\ell)+C_{ij\ell}(\alpha,\beta_\ell)\right]\right)\left(\prod_{s \in \mathbb{I}_{ib_s}, s \neq m} \left[ J_{js}(\alpha,\beta_s)+K_{js}(\alpha,\beta_s)\right]\right)\nonumber\\
		&&\times  \left(J_{jm}(\alpha,\beta_m)IG_{\alpha,G_{jm}(\beta_m)}(h_{im})+K_{jm}(\alpha,\beta_m)IG_{\alpha+1,G_{jm}(\beta_m)}(h_{im})\right).\nonumber\\
	\end{eqnarray*}
It follows therefore that	
	\begin{eqnarray}\label{h3}
		\widetilde{h}_{im}&=&\frac{1}{D_{i}(\alpha,\boldsymbol{\beta})}\sum_{j=1,j\neq i}^{n} \left(\prod_{k \in \mathbb{I}_{ia_k}}\left[ F_{jk}(\alpha,\beta_k)+H_{jk}(\alpha,\beta_k)\right]\right) \nonumber\\
		&&\times \left(\prod_{\ell \in \mathbb{I}_{i}^c} \left[ A_{ij\ell}(\alpha,\beta_\ell)+C_{ij\ell}(\alpha,\beta_\ell)\right]\right)\left(\prod_{s \in \mathbb{I}_{ib_s}, s \neq m} \left[ J_{js}(\alpha,\beta_s)+K_{js}(\alpha,\beta_s)\right]\right)\nonumber\\
		&&\times  \left(\frac{J_{jm}(\alpha,\beta_m)}{\alpha-1}+\frac{K_{jm}(\alpha,\beta_m)}{\alpha}\right)\frac{G_{ijm}(\beta_m)}{J_{jm}(\alpha,\beta_m)+K_{jm}(\alpha,\beta_m)}
	\end{eqnarray}
Finally, combining \eqref{h1}, \eqref{h2} and \eqref{h3} yields the closed expression of Part (ii), which completes the proof.	
\end{proof}

\section*{Acknowledgments}

The authors dedicate this paper to Professor Faustin Archange Touadera for his 66th birthday, with our gratitute for his availability and support. Part of this work was undertaken while the first author was at \emph{Laboratoire de math\'ematiques de Besan\c con} as a visiting scientist, with financial support from \emph{Universit\'e Thomas SANKARA}. This work has been conducted within the frame of EIPHI graduate school (contrat "ANR-17-EURE-0002") of the second author.


\end{document}